\documentclass[12pt, reqno]{amsart}
\usepackage{amsmath, amsthm, amscd, amsfonts, amssymb, graphicx, color}
\usepackage[bookmarksnumbered, colorlinks, plainpages]{hyperref}
\hypersetup{colorlinks=true,linkcolor=red, anchorcolor=green, citecolor=cyan, urlcolor=red, filecolor=magenta, pdftoolbar=true}

\newtheorem{theorem}{Theorem}[section]
\newtheorem{lemma}[theorem]{Lemma}
\newtheorem{proposition}[theorem]{Proposition}
\newtheorem{corollary}[theorem]{Corollary}
\theoremstyle{definition}

\theoremstyle{remark}
\newtheorem{remark}[theorem]{Remark}
\numberwithin{equation}{section}

\newcommand{\be}{\begin{equation}}
\newcommand{\ee}{\end{equation}}

\newcommand{\cM}{{\mathcal M}}
\newcommand{\NN}{\mathbb{N}}

\begin{document}
\setcounter{page}{1}

\title[Essential joint and essential generalized spectral radius]{Inequalities on the essential joint and essential generalized spectral radius}

\author{ Brian Lins$^{1}$, Aljo\v{s}a Peperko$^{2,3*}$}
\date{\thanks{
Hampsdey-Sydney College$^1$,
Virginia, USA,
email: lins.brian@gmail.com
  \\
Faculty of Mechanical Engineering$^2$,
University of Ljubljana,
A\v{s}ker\v{c}eva 6,
SI-1000 Ljubljana, Slovenia,\\
%{\it and}
Institute of Mathematics, Physics and Mechanics$^3$,
%University of Ljubljana \\
Jadranska 19,
SI-1000 Ljubljana, Slovenia \\
e-mail:   aljosa.peperko@fs.uni-lj.si\\
*Corresponding author
} \today}

%\dedicatory{This paper is dedicated to Professor ABCD}

\subjclass[2020]{47A10, 47B65, 47B34, 15A42, 15A60, 15B48}
%Primary 47B65; Secondary 15A42, 15A60, 15A18, 15A45, 15B48,}

\keywords{weighted Hadamard-Schur geometric mean; Hadamard-Schur product; essential spectral radius; Haussdorf measure of noncompactness; joint and generalized spectral radius; 
 positive kernel operators; non-negative matrices; bounded sets of operators}

%\date{Received: xxxxxx; Revised: yyyyyy; Accepted: zzzzzz.
%\newline \indent $^{*}$ Corresponding author}

\begin{abstract}
We prove new inequalities for the essential generalized and the essential joint spectral radius of Hadamard (Schur) weighted geometric means of bounded sets of infinite nonnegative matrices that define operators on suitable  Banach  sequence spaces and of bounded sets of positive kernel operators on $L^2$. To our knowledge the obtained inequalities are new even in the case of singelton sets.
\end{abstract} \maketitle

\section{Introduction}

In \cite{Zh09}, X. Zhan conjectured that, for non-negative $N\times N$ matrices $A$ and $B$, the spectral radius $\rho (A\circ B)$ of the Hadamard product satisfies
\be
\rho (A\circ B) \le \rho (AB),
\label{qu}
\ee
where $AB$ denotes the usual matrix product of $A$ and $B$. This conjecture was confirmed by K.M.R. Audenaert in \cite{Au10}
by proving
\be
\rho (A\circ B) \le \rho ((A\circ A)(B\circ B)) ^{\frac{1}{2}}\le \rho (AB).
\label{Aud}
\ee
These inequalities were established via a trace description of the spectral radius.
%via a trace description of the spectral radius.
Soon after, inequality (\ref{qu}) was reproved, generalized and refined in different ways by
several authors (\cite{HZ10, Hu11, S11, Sc11, P12, CZ15, DP16, P17, P17+, BP21, BP22b, B23+}).
% by proving
%\be
%\rho (A\circ B) \le \rho ^{\frac{1}{2}}((A\circ A)(B\circ B))\le \rho (AB).
%\label{Aud}
%\ee
%These inequalities were established via a trace description of the spectral radius. Using the fact that the Hadamard product is a principal submatrix of the Kronecker product, R.A. Horn and F. Zhang  proved in \cite{HZ10} the inequalities
%\be
%\rho (A\circ B) \le \rho ^{\frac{1}{2}}(AB\circ BA)\le \rho (AB)
%\label{HZ}
%\ee
%and also the right-hand side inequality in (\ref{Aud}).
Using the fact that the Hadamard product is a principal submatrix of
the Kronecker product, R.A. Horn and F. Zhang  proved in \cite{HZ10} the inequalities
\be
\rho (A\circ B) \le \rho (AB\circ BA)^{\frac{1}{2}}\le \rho (AB).
\label{HZ}
\ee
Applying the techniques of \cite{HZ10}, Z.  Huang proved that
%Applying a fact that the Hadamard product is a principal submatrix of the Kronecker product (i.e., by applying the technique used by R.A. Horn and F. Zhang  of \cite{HZ10}), Z. Huang proved that
\be
\rho (A_1 \circ A_2 \circ \cdots \circ A_m) \le \rho (A_1 A_2 \cdots A_m)
\label{Hu}
\ee
for $n\times n$ non-negative matrices $A_1, A_2, \cdots, A_m$ (see \cite{Hu11}).  A.R. Schep was the first one to observe that the results from \cite{DP05} and \cite{P06} are applicable in this context (see \cite{S11} and \cite{Sc11}). He extended inequalities (\ref{Aud}) and (\ref{HZ}) to non-negative matrices that define bounded
operators on sequence spaces (in particular on $l^p$ spaces, $1\le p <\infty$) and proved in
\cite[Theorem 2.7]{S11}  that
\be
\rho (A\circ B) \le \rho ((A\circ A)(B\circ B))^{\frac{1}{2}}\le  \rho (AB\circ AB)^{\frac{1}{2}}\le \rho (AB)
\label{Sproved}
\ee
(note that there was an error in the statement of \cite[Theorem 2.7]{S11}, which was corrected in \cite{Sc11} and \cite{P12}).
In  \cite{P12}, the second author of the current paper extended the inequality (\ref{Hu}) to non-negative matrices that define bounded
operators on Banach sequence spaces (see below for the exact definitions) and proved that the inequalities %\cite{P12} for exact definitions) and proved that the inequalities
\be
\rho (A\circ B) \le \rho ((A\circ A)(B\circ B))^{\frac{1}{2}}\le \rho(AB \circ AB)^{\frac{\beta}{2}} \rho(BA \circ BA)^{\frac{1-\beta}{2}} \le \rho (AB)
\label{P1}
\ee
and
\be
\rho (A\circ B) \le \rho (AB\circ BA)^{\frac{1}{2}}\le  \rho(AB \circ AB)^{\frac{1}{4}} \rho(BA \circ BA)^{\frac{1}{4}} \le  \rho (AB).
\label{P2}
\ee
hold, where $\beta \in [0,1]$. Moreover, % in \cite[Theorem 3.16]{P12}
 he generalized these inequalities to the setting of the generalized and the joint spectral radius of bounded sets of such non-negative matrices. % and proved several refinements in the case $m=2$.
%In the proofs certain results on the Hadamard product from \cite{DP05} and \cite{P06} were used.

%A related inequality
%for $n\times n$ non-negative matrices was shown in \cite{EJS88}:
%\be
%\rho (A_1 \circ A_2 \circ \cdots \circ A_m) \le \rho (A_1) \rho (A_2) \cdots \rho (A_m ) .
%\label{EJS}
%\ee

%Earlier, A.R. Schep was the first one to observe that the results  \cite{DP05} and \cite{P06} are applicable in this context (see \cite{S11} and \cite{Sc11}).
In \cite[Theorem 2.8]{S11},  A.R. Schep proved
 %In particular, in \cite[Theorem 2.8]{S11}  he proved
 that the inequality
\be
\rho \left(A ^{\left( \frac{1}{2} \right)} \circ B  ^{\left( \frac{1}{2} \right)} \right) \le \rho (AB) ^{\frac{1}{2}}
\label{Schep}
\ee
holds for positive kernel operators on $L^p$ spaces. Here $A ^{\left( \frac{1}{2} \right)} \circ B  ^{\left( \frac{1}{2}\right)} $ denotes the Hadamard geometric mean of operators $A$ and $B$. 
%In \cite[Theorem 3.1]{DP16}, R. Drnov\v{s}ek and the second author, generalized this inequality and proved that the inequality
%\be
%\rho \left(A_1^{\left(\frac{1}{m}\right)} \circ A_2^{\left(\frac{1}{m}\right)} \circ \cdots \circ
%A_m^{\left(\frac{1}{m}\right)}\right)   \le \rho (A_1 A_2 \cdots A_m)^{\frac{1}{m}}
%\label{genHuBfs}
%\ee
%holds for positive kernel operators $A_1, \ldots, A_m$ on an arbitrary Banach function space $L$. 
% In \cite{P17+}, the second author refined (\ref{genHuBfs}) and showed that the inequalities
 R. Drnov\v{s}ek and the second author (see \cite{DP16, P17+}), generalized this inequality and proved that the inequalities
$$\rho \left(A_1^{\left(\frac{1}{m}\right)} \circ A_2^{\left(\frac{1}{m}\right)} \circ \cdots \circ
A_m^{\left(\frac{1}{m}\right)}\right) \;\;\;\;\;\;\;\;\;\;\;\;\;\;\;\;\;\;\;\;\;\;\;\;\;\;\;\;\;\;$$
\be
\le \rho  \left(P_1^{\left(\frac{1}{m}\right)} \circ P_2^{\left(\frac{1}{m}\right)} \circ \cdots \circ
P_m^{\left(\frac{1}{m}\right)}\right) ^{\frac{1}{m}}  \le \rho (A_1 A_2 \cdots A_m) ^{\frac{1}{m}} .
\label{genHuBfsP}
\ee
hold  for positive kernel operators $A_1, \ldots, A_m$ on an arbitrary Banach function space $L$, where  $P_j = A_j \ldots A_m A_1 \ldots A_{j-1}$ for $j=1,\ldots , m$.  Formally, here and throughout the article $A_{j-1}=I$ for $j=1$ (eventhough $I$ might not be a well defined kernel operator). The second author proved further in \cite[Theorem 4.4, (4.8)]{P17+}) that in the $L^2$ case  it holds
\be
\|A^{(\frac{1}{2})} \circ  B^{(\frac{1}{2})} \| \le \rho  \left ( (A^* B ) ^{(\frac{1}{2})}\circ (B ^* A)^{(\frac{1}{2})}\right)^{\frac{1}{2}} \le \rho  (A^* B )^{\frac{1}{2}}= \rho  (A B^*)^{\frac{1}{2}},
\label{Pep19}
\ee
where $\|\cdot\|$ denotes the operator norm. 
In \cite[Theorem 3.2]{P19}, the second author showed that (\ref{genHuBfsP}) (and thus also (\ref{Schep})) holds also for the essential radius $\rho _{ess}$ under the additional condition that $L$ and its Banach dual $L^*$ have order continuous norms.
%In particular, the following kernel version of  (\ref{HZ}) holds:
%\be
%\rho \left(A^{(\frac{1}{2})} \circ B^{(\frac{1}{2})} \right) \le  \rho \left((AB)^{(\frac{1}{2})} \circ (BA)^{(\frac{1}{2})}\right)^{\frac{1}{2}} \le \rho(AB) ^{\frac{1}{2}}.
%\label{sch_refin}
%\ee
Several additional closely related results, generalizations and refinements of the above results were obtained in \cite{P19, Zh18, P21, BP21, BP22b, B23+}. However, it remained unclear whether the analogues of inequalities
(\ref{qu})-(\ref{P2}) are valid for the essential spectral radius of infinite nonnegative matrices that define operators on e.g. $l^2$ and whether an analogue of (\ref{Pep19}) is valid for a suitable measure of non-compactness and for the essential spectral radius of positive kernel operators on $L^2$.  In this paper (as a very special case of our results) we positively answer these questions (see Corollary \ref{essential_Aud}, Theorem \ref{kathyth2_ess} and Corollary \ref{ess_star} below).
%In \cite[Theorem 3.4]{P17} and \cite[Theorem 3.5]{P19}, the second author generalized inequalities (\ref{genHuBfs}) and (\ref{sch_refin}) and their essential version to the setting of the generalized and the joint spectral radius (and their essential versions) of bounded sets of positive kernel operators on a Banach function space (see also Theorems \ref{powers} and  %\ref{genHuDP}
%\ref{refin} below). %As already pointed out in \cite[Remark 3.4]{P17+}, the inequalities (\ref{genHuBfsP}) can also be deduced from the proof of \cite[Theorem 3.4]{P17}.

The rest of the article is organized in the following way. In Section 2 we recall definitions and results that we will use in our proofs. % and we present our results in Section 3.
In Section 3 we prove the key results (Theorems \ref{ess_needed}, \ref{special_case_ess} and \ref{thbegin_ess}) on the Haussdorf measure of noncompactness and the essential spectral radius of ordinary products of Hadamard powers and ordinary products of Hadamard weighted geometric means of infinite nonnegative matrices that define operators on suitable Banach sequence  spaces. These results are essential analogues of the known results for the operator norm and the spectral radius. By combining ideas of proofs from previously known results we prove in Theorem \ref{finally_ess} an extension of these results to the essential joint and essential generalized spectral radius of bounded sets of infinite nonnegative matrices.  In Section 4 we apply these results to obtain several essential analogues of known results on sums of Hadamard weighted geometric means, weighted geometric symmetrizations and Hadamard products of  bounded sets of infinite nonnegative matrices (Theorems \ref{finally2_ess}, \ref{sym_matrices}, \ref{kathyth2_ess} and Corollary \ref{P12_cor}). In Corollary \ref{essential_Aud} we obtain the essential versions of (\ref{P1}) and (\ref{P2}), while the essential version of (\ref{Hu}) is a very special case of Theorem \ref{kathyth2_ess}.  In Section 5 we prove new essential results for operators on Hilbert spaces. In Corollary \ref{ess_star} we prove the essential version of (\ref{Pep19}). We conclude the article by obtaining essential versions of several recent results from \cite{B23+}.

\section{Preliminaries}
\vspace{1mm}

Let $\mu$ be a $\sigma$-finite positive measure on a $\sigma$-algebra $\cM$ of subsets of a non-void set $X$.
Let $M(X,\mu)$ be the vector space of all equivalence classes of (almost everywhere equal)
complex measurable functions on $X$. A Banach space $L \subseteq M(X,\mu)$ is
called a {\it Banach function space} if $f \in L$, $g \in M(X,\mu)$,
and $|g| \le |f|$ imply that $g \in L$ and $\|g\| \le \|f\|$. Throughout the article, it is assumed that  $X$ is the carrier of $L$, that is, there is no subset $Y$ of $X$ of
 strictly positive measure with the property that $f = 0$ a.e. on $Y$ for all $f \in L$ (see \cite{Za83}).

%Observe that a Banach sequence space is a Banach function space over a measure space $(R, \mu)$,
%where $\mu$ denotes the counting measure on $R$ (and for $L\in \mathcal{L}$ the set $R$ is the carrier of $L$).
%Throughout the article, it is assumed that $X$ is the carrier of $L$, that is, there is no subset $Y$ of $X$ of
%strictly positive measure with the property that $f = 0$ a.e. on $Y$ for all $f \in L$ (see \cite{Za83}).
 Let $R$ denote the set $\{1, \ldots, N\}$ for some $N \in \NN$ or the set $\NN$ of all natural numbers.
Let $S(R)$ be the vector lattice of all complex sequences $(x_n)_{n\in R}$.
%Let $M_0 (X,\mu)$ be the vector lattice of all equivalence classes of (almost everywhere equal)
%complex $\mu$-measurable functions on $X$.
A Banach space $L \subseteq S(R)$ is called a {\it Banach sequence space} if $x \in S(R)$, $y \in L$
and $|x| \le |y|$ imply that $x \in L$ and $\|x\|_L \le \|y\|_L$. Observe that a Banach sequence space is a Banach function space over a measure space $(R, \mu)$,
where $\mu$ denotes the counting measure on $R$. Denote by $\mathcal{L}$ the collection of all Banach sequence spaces
$L$ satisfying the property that $e_n = \chi_{\{n\}} \in L$ and
$\|e_n\|_L=1$ for all $n \in R$. For $L\in \mathcal{L}$ the set $R$ is the carrier of $L$.
%We say that $L \in \mathcal{L}$ has property (*) if for all $x\in L$ it holds $x(i) \to 0$ as $i \to \infty$.

Standard examples of Banach sequence spaces are Euclidean spaces, $l^p$ spaces for $1\le p \le \infty$,  the space $c_0\in \mathcal{L}$
of all null convergent sequences  (equipped with the usual norms and the counting measure), while standard examples of Banach function spaces are  the well-known spaces $L^p (X,\mu)$ ($1\le p \le \infty$) and other less known examples such as Orlicz, Lorentz,  Marcinkiewicz  and more general  rearrangement-invariant spaces (see e.g. \cite{BS88, CR07, KM99} and the references cited there), which are important e.g. in interpolation theory and in the theory of partial differential equations.
 Recall that the cartesian product $L=E\times F$
of Banach function spaces is again a Banach function space, equipped with the norm
$\|(f, g)\|_L=\max \{\|f\|_E, \|g\|_F\}$. 

If $\{f_n\}_{n\in \mathbb{N}} \subset M(X,\mu)$ is a decreasing sequence and
$f=\inf\{f_n \in M(X,\mu): n \in \NN \}$, then we write $f_n \downarrow f$. %Similarly we define $x_n \uparrow x$.
A Banach function space $L$ has an {\it order continuous norm}, if $0\le f_n \downarrow 0$
implies $\|f_n\|_L \to 0$ as $n \to \infty$. It is well known that spaces $L^p  (X,\mu)$, $1\le p< \infty$, have order continuous
norm. Moreover, the norm of any reflexive Banach function space is
order continuous. %However, the norm $\|x\|_{\infty}=\sup\{|x_i|:i\in \NN\}$ on the space $l^{\infty}$ of all bounded sequences or on
%the space $c$ of all convergent sequences, is not order continuous. Indeed, let
%$\{x_n\}_{n=1} ^{\infty} \subset c$, where $x_n =\{a_{m,n}\}$ and $a_{m,n} =0$ for $m \le n$ and
 %$a_{m,n} =1$ for $m > n$. Then $x_n \downarrow 0$, but $\|x_n\|_{\infty}=1$ for all $n \in \NN$.
In particular, we will be interested in  Banach function spaces $L$ such that $L$ and its Banach dual space $L^*$ have order continuous norms. Examples of such spaces are $L^p  (X,\mu)$, $1< p< \infty$, while the space
$L=c_0$ %(equipped with the norm $\|\cdot\|_{\infty}$)
 %of all null convergent sequences
is an example of a non-reflexive Banach sequence space, such that $L$ and  $L^*=l^1$ have order continuous
norms. %The spaces $L=l^p$ spaces for $1< p < \infty$ and the space $L=c_0$ are examples of Banach sequence spaces from $\mathcal{L}$ such that $L$ has property $(*)$ and  such that $L$ and $L^*$ have order continuous norms. 
  
By an {\it operator} on a Banach function space $L$ we always mean a linear
operator on $L$.  An operator $A$ on $L$ is said to be {\it positive}
if it maps nonnegative functions to nonnegative ones, i.e., $AL_+ \subset L_+$, where $L_+$ denotes the positive cone $L_+ =\{f\in L : f\ge 0 \; \mathrm{a.e.}\}$.
Given operators $A$ and $B$ on $L$, we write $A \ge B$ if the operator $A - B$ is positive.

Recall that a positive  operator $A$ is always bounded, i.e., its operator norm
\be
\|A\|=\sup\{\|Ax\|_L : x\in L, \|x\|_L \le 1\}=\sup\{\|Ax\|_L : x\in L_+, \|x\|_L \le 1\}
\label{equiv_op}
\ee
is finite.
Also, its spectral radius $\rho (A)$ is always contained in the spectrum.

An operator $A$ on a Banach function space $L$ is called a {\it kernel operator} if
there exists a $\mu \times \mu$-measurable function
$a(x,y)$ on $X \times X$ such that, for all $f \in L$ and for almost all $x \in X$,
$$ \int_X |a(x,y) f(y)| \, d\mu(y) < \infty \ \ \ {\rm and} \ \
   (Af)(x) = \int_X a(x,y) f(y) \, d\mu(y)  .$$
One can check that a kernel operator $A$ is positive iff
its kernel $a$ is non-negative almost everywhere.

Let $L$ be a Banach function space such that $L$ and $L^*$ have order
continuous norms and let $A$ and $B$ be  positive kernel operators on $L$. By $\gamma (A)$ we denote the Hausdorff measure of
non-compactness of $A$, i.e.,
$$\gamma (A) = \inf\left\{ \delta >0 : \;\; \mathrm{there}\;\; \mathrm{is} \;\; \mathrm{a}\;\; \mathrm{finite}\;\; M \subset L \;\;\mathrm{such} \;\; \mathrm{that} \;\; A(D_L) \subset M + \delta D_L  \right\},$$
where $D_L =\{f\in L : \|f\|_L \le 1\}$. Then $\gamma (A) \le \|A\|$, $\gamma (A+B) \le \gamma (A) + \gamma (B)$, $\gamma(AB) \le \gamma (A)\gamma (B)$ and $\gamma (\alpha A) =\alpha \gamma (A)$ for $\alpha \ge 0$. Also
$0 \le A\le B$  implies $\gamma (A) \le \gamma (B)$ (see e.g. \cite[Corollary 4.3.7 and Corollary 3.7.3]{Me91}). Let $\rho _{ess} (A)$ denote the essential spectral radius of $A$, i.e., the spectral radius of the Calkin image of $A$ in the Calkin algebra. Then
\be
 \rho _{ess} (A) =\lim _{j \to \infty} \gamma (A^j)^{1/j}=\inf _{j \in \NN} \gamma (A^j)^{1/j}
\label{esslim=inf}
\ee
and $\rho _{ess} (A) \le \gamma (A)$. Recall that if $L=L^2(X, \mu)$, then $\gamma (A^*) = \gamma (A)$ and $\rho _{ess} (A^*)=\rho _{ess} (A)  $, where $A^*$ denotes the adjoint of $A$ (see e.g. \cite[Proposition 4.3.3, Theorems 4.3.6 and 4.3.13 and Corollary 3.7.3]{Me91}, \cite[Theorem 1]{Nussbaum70}). Note that equalities (\ref{esslim=inf}) and  $\rho _{ess} (A^*)=\rho _{ess} (A)  $ are valid for any bounded operator $A$ on a given complex Banach space $L$ (see e.g. \cite[Theorem 4.3.13 and Proposition 4.3.11]{Me91}, \cite[Theorem 1]{Nussbaum70}).

% $M_0(X,\mu)$, such that $X$ is the carrier of $L$ and let $L$ and $L^*$ have order
%continuous norms. Let $\phi$, $\alpha _1, \ldots , \alpha _m$ and $C$ be as in Example \ref{locrad}. Define
%$h(k)=\beta (K)$ for $k\in C$  and $h (k)=\infty$ for $k \notin C$. Here $\beta (K)$ denotes the measure of
%non-compactness, i.e.,
%$$\beta (K) = \inf\left\{ \delta >0 : \;\; \mathrm{there}\;\; \mathrm{is} \;\; \mathrm{a}\;\; \mathrm{finite}\;\; M \subset L \;\;\mathrm{such} \;\; \mathrm{that} \;\; K(B_L) \subset M + \delta B_L  \right\},$$
%where $B_L =\{f\in L : \|f\|_L \le 1\}$.
%Then $h:M(X \times X)_+ \to [0, \infty]$ is a function seminorm (see e.g. \cite[Corollary 4.3.7 and Corollary 3.7.3]{Me91}),
%$C = C_h$ and for $k \in C$ we have
%$$r_h (k) = \lim _{j \to \infty} \beta (K^j)^{1/j}=\inf _{j \in \NN} \beta (K^j)^{1/j} =r_{ess} (K)$$
%(the essential spectral radius of $K$; see e.g. \cite[Theorem 4.3.13]{Me91}).

%Observe that (finite or infinite) non-negative matrices, that define operators on Banach sequence spaces, are a special case of positive kernel operators.
%(see e.g. \cite{P12}, \cite{DP16}, \cite{DP10}, \cite{P11}, \cite{BP21}, and the references cited there).
It is well-known that kernel operators play a very important, often even central, role in a variety of applications from differential and integro-differential equations, problems from physics
(in particular from thermodynamics), engineering, statistical and economic models, etc (see e.g. \cite{J82, P19} %\cite{BP03}, \cite{LL05}, \cite{DLR13} % , \cite{HR12}, \cite{AH83}, \cite{KS82}, \cite{AK71},  \cite{W75}
and the references cited there).
For the theory of Banach function spaces and more general Banach lattices we refer the reader to the books \cite{Za83, BS88, AA02, AB85, Me91}.

Let $A$ and $B$ be positive kernel operators on a Banach function space $L$ with kernels $a$ and $b$ respectively,
and $\alpha \ge 0$.
The \textit{Hadamard (or Schur) product} $A \circ B$ of $A$ and $B$ is the kernel operator
with kernel equal to $a(x,y)b(x,y)$ at point $(x,y) \in X \times X$ which can be defined (in general)
only on some order ideal of $L$. Similarly, the \textit{Hadamard (or Schur) power}
$A^{(\alpha)}$ of $A$ is the kernel operator with kernel equal to $(a(x, y))^{\alpha}$
at point $(x,y) \in X \times X$ which can be defined only on some order ideal of $L$.

Let $A_1 ,\ldots, A_m$ be positive kernel operators on a Banach function space $L$,
and $\alpha _1, \ldots, \alpha _m$ positive numbers such that $\sum_{j=1}^m \alpha _j = 1$.
Then the {\it  Hadamard weighted geometric mean}
$A = A_1 ^{( \alpha _1)} \circ A_2 ^{(\alpha _2)} \circ \cdots \circ A_m ^{(\alpha _m)}$ of
the operators $A_1 ,\ldots, A_m$ is a positive kernel operator defined
on the whole space $L$, since $A \le \alpha _1 A_1 + \alpha _2 A_2 + \ldots + \alpha _m A_m$ by the inequality between the weighted arithmetic and geometric means.

A matrix $A=[a_{ij}]_{i,j\in R}$ is called {\it nonnegative} if $a_{ij}\ge 0$ for all $i, j \in R$.
% We will denote a nonnegative matrix by $K\ge 0$.
For notational convenience, we sometimes write $a(i,j)$ instead of $a_{ij}$.

We say that a nonnegative matrix $A$ defines an operator on $L$ if $Ax \in L$ for all $x\in L$, where
$(Ax)_i = \sum _{j \in R}a_{ij}x_j$. Then $Ax \in L_+$ for all $x\in L_+$ and
so $A$ defines a positive kernel operator on $L$.

 Let us recall  the following result, which was proved in \cite[Theorem 2.2]{DP05} and
\cite[Theorem 5.1 and Example 3.7]{P06} (see also e.g. \cite[Theorem 2.1]{P17}).

\begin{theorem}
\label{thbegin}
Let $\{A_{i j}\}_{i=1, j=1}^{k, m}$ be positive kernel operators on a Banach function space $L$ and  $\alpha _1$, $\alpha _2$,..., $\alpha _m$  positive numbers.

\noindent (i) If %$\alpha _1$, $\alpha _2$,..., $\alpha _m$ are positive numbers
%such that 
$\sum_{j=1}^{m} \alpha _j = 1$, then the positive kernel operator
\be
A:= \left(A_{1 1}^{(\alpha _1)} \circ \cdots \circ A_{1 m}^{(\alpha _m)}\right) \ldots \left(A_{k 1}^{(\alpha _1)} \circ \cdots \circ A_{k m}^{(\alpha _m)} \right)
\label{osnovno}
\ee
%then the inequalities%(\ref{basic2}), (\ref{norm2}) and (\ref{spectral2}) hold.
satisfies the following inequalities
\begin{equation}
%\label{basic2}
A \le
(A_{1 1} \cdots  A_{k 1})^{(\alpha _1)} \circ \cdots
\circ (A_{1 m} \cdots A_{k m})^{(\alpha _m)} , \\
\label{norm2}
\end{equation}
\begin{eqnarray}
\nonumber
\left\|A \right\| &\le &\left\|(A_{1 1} \cdots  A_{k 1})^{(\alpha _1)} \circ \cdots
\circ (A_{1 m} \cdots A_{k m})^{(\alpha _m)} \right\|\\
&\le&\left\|A_{1 1} \cdots  A_{k 1}\right\|^{\alpha _1}\cdots\left\|A_{1 m} \cdots  A_{k m}\right\|^{\alpha _m}
\label{spectral2}
\end{eqnarray}
\begin{eqnarray}
\nonumber
\rho \left(A \right)& \le &\rho \left((A_{1 1} \cdots  A_{k 1})^{(\alpha _1)} \circ \cdots
\circ (A_{1 m} \cdots A_{k m})^{(\alpha _m)}\right)\\
&\le&\rho \left( A_{1 1} \cdots  A_{k 1} \right)^{\alpha _1} \cdots
\rho \left( A_{1 m} \cdots A_{k m}\right)^{\alpha _m} .
\label{tri}
\end{eqnarray}
%\end{eqnarray}
If, in addition, $L$ and $L^*$ have order continuous norms, then% the inequalities (\ref{norm2}), (\ref{spectral2}) and (\ref{tri}) hold.
\begin{eqnarray}
\nonumber
\gamma (A) & \le &\gamma \left((A_{1 1} \cdots  A_{k 1})^{(\alpha _1)} \circ \cdots
\circ (A_{1 m} \cdots A_{k m})^{(\alpha _m)}\right)\\
 &\le &
 \label{meas_noncomp}
\gamma (A_{1 1} \cdots  A_{k 1})^{\alpha _1} \cdots \gamma(A_{1 m} \cdots A_{k m})^{\alpha _m}, \\
\nonumber
\rho _{ess} \left(A \right) & \le &\rho _{ess} \left((A_{1 1} \cdots  A_{k 1})^{(\alpha _1)} \circ \cdots
\circ (A_{1 m} \cdots A_{k m})^{(\alpha _m)}\right)\\
&\le  &
\label{ess_spectral}
\rho _{ess} \left( A_{1 1} \cdots  A_{k 1} \right)^{\alpha _1} \cdots
\rho _{ess} \left( A_{1 m} \cdots A_{k m}\right)^{\alpha _m} .
\end{eqnarray}

%\label{DPBfs}
\noindent(ii) If $L\in\mathcal L$, $\sum_{j=1}^{m} \alpha _j\ge 1$ and $\{A_{i j}\}_{i=1, j=1}^{k, m}$ are nonnegative matrices that define positive operators on $L$, then $A$ from (\ref{osnovno}) defines a  positive operator on $L$ and the inequalities (\ref{norm2}), (\ref{spectral2}) and (\ref{tri}) hold.
\label{DBPfs}
\end{theorem}
The following result is a special case  of Theorem \ref{DBPfs}.
\begin{theorem}
\label{special_case}
Let $A_1 ,\ldots, A_m$ be positive kernel operators on a Banach function space  $L$
and $\alpha _1, \ldots, \alpha _m$ positive numbers. % such that $\sum_{j=1}^m \alpha _j = 1$.

\noindent (i) If $\sum_{j=1}^m \alpha _j = 1$, then 
\be
 \|A_1 ^{( \alpha _1)} \circ A_2 ^{(\alpha _2)} \circ \cdots \circ A_m ^{(\alpha _m)} \| \le
  \|A_1\|^{ \alpha _1}  \|A_2\|^{\alpha _2} \cdots \|A_m\|^{\alpha _m}
\label{gl1nrm}
\ee
and
\be
 \rho(A_1 ^{( \alpha _1)} \circ A_2 ^{(\alpha _2)} \circ \cdots \circ A_m ^{(\alpha _m)} ) \le
\rho(A_1)^{ \alpha _1} \, \rho(A_2)^{\alpha _2} \cdots \rho(A_m)^{\alpha _m} .
%L=M
\label{gl1vecr}
\ee
If, in addition, $L$ and $L^*$ have order continuous norms, then
\be
 \gamma (A_1 ^{( \alpha _1)} \circ A_2 ^{(\alpha _2)} \circ \cdots \circ A_m ^{(\alpha _m)} )\le
  \gamma(A_1)^{ \alpha _1}  \gamma(A_2)^{\alpha _2} \cdots \gamma(A_m)^{\alpha _m}
\label{gl1meas_nonc}
\ee
and
\be
 \rho _{ess}(A_1 ^{( \alpha _1)} \circ A_2 ^{(\alpha _2)} \circ \cdots \circ A_m ^{(\alpha _m)} ) \le
\rho _{ess }(A_1)^{ \alpha _1} \, \rho _{ess}(A_2)^{\alpha _2} \cdots \rho _{ess}(A_m)^{\alpha _m} .
%L=M
\label{gl1vecress}
\ee
\noindent(ii) If $L\in\mathcal L$, $\sum_{j=1}^{m} \alpha _j\ge 1$ and if  $A_1 ,\ldots, A_m$ are nonnegative matrices that define positive operators on $L$, then $A_1 ^{( \alpha _1)} \circ A_2 ^{(\alpha _2)} \circ \cdots \circ A_m ^{(\alpha _m)}$ defines a positive operator on $L$ and (\ref{gl1nrm}) and (\ref{gl1vecr}) hold.

\noindent (iii)  If $L\in\mathcal L$, $t\ge1$ and if $A, A_1 ,\ldots, A_m$ are nonnegative matrices that define operators on $L$,    then $A^{(t)}$ % $A_1^{(\alpha_1)}\circ A_2^{(\alpha_2)}\circ\cdots\circ A_m^{(\alpha_m)}$
 defines an  operator on L and the following inequalities hold
\be
\;\;\;\;\;A_1^{(t)}\cdots A_m^{(t)}\le(A_1\cdots A_m)^{(t)},
\label{gl1t}
\ee
\be
\rho(A_1^{(t)}\cdots A_m^{(t)})\le\rho(A_1\cdots A_m)^{t},
\label{gl1nt}
\ee
\be
\|A_1^{(t)}\cdots A_m^{(t)}\|\le\|A_1\cdots A_m\|^{t}.
\label{gl1vecrt}
\ee
\end{theorem}
The following result was proved in \cite[Corollary 2.10]{P21}.
\begin{theorem}  Given $L\in \mathcal{L}$, let $A$ be a nonnegative matrix that defines an operator on $L$ and let $t\ge 1$. Then 
\begin{eqnarray}
A^{(t)} &\le & \|A\|_{\infty} ^{t-1} A, \\
\label{norm_imp_t}
\|A^{(t)}\| &\le & \|A\|_{\infty} ^{t-1} \| A\|, \\
\label{dobra_t}
\rho(A^{(t)} )&\le & \|A\|_{\infty} ^{t-1}\rho( A). 
\end{eqnarray} 
If, in addition, $L$ and $L^*$ have order continuous norms, then
\begin{eqnarray}
\label{dobra_gamma}
\gamma (A^{(t)})&\le & \|A\|_{\infty} ^{t-1} \gamma(A), \\
\label{dobra_r_ess}
\rho_{ess}(A^{(t)} )&\le & \|A\|_{\infty} ^{t-1}\rho_{ess}(A). 
\end{eqnarray} 
%If, in addition, $L=l^2 (R)$, then
%\be
%w(K^{(t)} )\le  \|K\|_{\infty} ^{t-1}w( K). 
%\ee
\end{theorem}

%(\cite{EJS88}, \cite{DP05}, \cite{P06}, \cite{P11}, \cite{P12}, \cite{DP16}).

\bigskip

Let $\Sigma$ be a bounded set of bounded operators on a complex Banach space $L$.
%$n \times n$ complex matrices.
For $m \ge 1$, let
$$\Sigma ^m =\{A_1A_2 \cdots A_m : A_i \in \Sigma\}.$$
The generalized spectral radius of $\Sigma$ is defined by
\be
\rho (\Sigma)= \limsup _{m \to \infty} \;[\sup _{A \in \Sigma ^m} \rho (A)]^{1/m}
\label{genrho}
\ee
and is equal to
$$\rho (\Sigma)= \sup _{m \in \NN} \;[\sup _{A \in \Sigma ^m} \rho (A)]^{1/m}.$$
The joint spectral radius of $\Sigma$ is defined by
%It was shown in \cite{BW92} that $\rho (\Sigma)$ is equal to the joint spectral radius of $\Sigma$, i.e.,
\be
\hat{\rho}  (\Sigma)= \lim _{m \to \infty}[\sup _{A \in \Sigma ^m} \|A\|]^{1/m}.
\label{BW}
\ee
%where $\|\cdot\|$ is any vector norm on $\CC ^{n\times n}$.
Similarly, the generalized essential spectral radius of $\Sigma$ is defined by
\be
\rho _{ess} (\Sigma)= \limsup _{m \to \infty} \;[\sup _{A \in \Sigma ^m} \rho _{ess} (A)]^{1/m}
\label{genrhoess}
\ee
and is equal to
$$\rho _{ess} (\Sigma)= \sup _{m \in \NN} \;[\sup _{A \in \Sigma ^m} \rho _{ess} (A)]^{1/m}.$$
The joint essential  spectral radius of $\Sigma$ is defined by
%It was shown in \cite{BW92} that $\rho (\Sigma)$ is equal to the joint spectral radius of $\Sigma$, i.e.,
\be
\hat{\rho} _{ess}  (\Sigma)= \lim _{m \to \infty}[\sup _{A \in \Sigma ^m} \gamma (A)]^{1/m}.
\label{jointess}
\ee

It is well known that $\rho (\Sigma)= \hat{\rho}  (\Sigma)$ for a precompact nonempty set $\Sigma$ of compact operators on $L$ (see e.g. \cite{ShT00, ShT08, Mo}),
in particular for a bounded set of complex $n\times n$ matrices (see e.g. %\cite{BW92}, \cite{E95},
 \cite{SWP97, Dai11, MP12} and the references cited there).
This equality is called the Berger-Wang formula or also the
generalized spectral radius theorem. % (for an elegant proof in the finite dimensional case see \cite{Dai11}).
It is known that also the generalized Berger-Wang formula holds, i.e, that for any precompact nonempty  set $\Sigma$ of bounded operators on $L$ we have
$$\hat{\rho}  (\Sigma) = \max \{\rho (\Sigma), \hat{\rho} _{ess}  (\Sigma)\}$$
(see e.g.  \cite{ShT08, Mo, ShT00}). Observe also that it was proved in \cite{Mo} that in the definition of  $\hat{\rho} _{ess}  (\Sigma)$ one may replace the Haussdorf measure of noncompactness by several other seminorms, for instance it may be replaced by the essential norm.

In general $\rho (\Sigma)$ and $\hat{\rho}  (\Sigma)$ may differ even in the case of a bounded set $\Sigma$ of compact positive operators on $L$ (see \cite{SWP97} or also \cite{P17}).
 %as the
%following example from \cite{SWP97} shows. Let
%$\Sigma =\{A_1, A_2, \ldots \}$ be a bounded set of compact operators on $L=l^2$ defined by $A_k e_k =e_{k+1}, (k \in \NN) $ and
%$A_k e_j =0$ for $j\neq k$.
%%$$A_k e_k =e_{k+1}, \;\;\; k \in \NN ,$$
%%$$A_k e_j =0, \;\;\; j\neq k .$$
%Then $(A_{i_1}A_{i_2}\cdots A_{i_k})^2=0$ for arbitrary $k \in \NN$ and any subset $\{i_1, i_2, \ldots, i_k\} \subset \NN$.
%%$$(A_{i_1}A_{i_2}\cdots A_{i_k})^2=0.$$
%Thus $\rho(\Sigma)=0$. Since
%$$A_mA_{m-1} \cdots A_1 e_1 =e_{m+1}, \;\;\; m \in \NN , $$
%$$A_mA_{m-1} \cdots A_1 e_j =0, \;\;\; j\neq 1 ,$$
%we have
%$\hat{\rho}(\Sigma)\ge \limsup _{m \to \infty} \|A_m \cdots A_1\| ^{1/m}=1$
%and so $\rho(\Sigma) \neq \hat{\rho}(\Sigma)$.
Also, in \cite{Gui82} the reader can find an example of two positive non-compact weighted shifts $A$ and $B$ on $L=l^2$ such that $\rho(\{A,B\})=0 < \hat{\rho}(\{A,B\})$. As already noted in \cite{ShT00} also $\rho _{ess} (\Sigma)$ and $\hat{\rho} _{ess}  (\Sigma)$ may in general be different.

%Since then many different type of proofs of (\ref{BW}) were obtained (for references see e.g. \cite{L06}).
The theory of the generalized and the joint spectral radius has many important applications for instance to discrete and differential inclusions,
wavelets, invariant subspace theory
(see e.g. %\cite{BW92}, 
\cite{Dai11, Wi02, ShT00, ShT08} and the references cited there).
In particular, $\hat{\rho} (\Sigma)$ plays a central role in determining stability in convergence properties of discrete and differential inclusions. In this
theory the quantity $\log \hat{\rho} (\Sigma)$ is known as the maximal Lyapunov exponent (see e.g. \cite{Wi02}).

We will  use the following well known facts that hold for all $r \in \{\rho,  \hat{\rho}, \rho _{ess}, \hat{\rho} _{ess}  \}$:
\be
r (\Sigma  ^m) = r (\Sigma)^m \;\;\mathrm{and}\;\;
%\hat{\rho} (\Sigma  ^m) = \hat{\rho} (\Sigma)^m ,\;\;
r (\Psi \Sigma) = r (\Sigma\Psi)
\label{again}
\ee
%\;\;\mathrm{and}\;\; \hat{\rho} (\Psi \Sigma) = \hat{\rho} (\Sigma\Psi),$$
where $\Psi \Sigma =\{AB: A\in \Psi, B\in \Sigma\}$ and $m\in \NN$.

Let $\Psi _1, \ldots , \Psi _m$ be bounded sets of positive kernel operators on a Banach function space $L$ and let $\alpha _1, \ldots \alpha _m$ be positive numbers such that
$\sum _{i=1} ^m \alpha _i = 1$. Then the bounded set of positive kernel operators on $L$, defined by
$$\Psi _1 ^{( \alpha _1)} \circ \cdots \circ \Psi _m ^{(\alpha _m)}=\{ A_1 ^{( \alpha _1)} \circ \cdots \circ A _m ^{(\alpha _m)}: A_1\in \Psi _1, \ldots, A_m \in \Psi _m \},$$
is called the {\it weighted Hadamard (Schur) geometric mean} of sets $\Psi _1, \ldots , \Psi _m$. The set
$\Psi _1 ^{(\frac{1}{m})} \circ \cdots \circ \Psi _m ^{(\frac{1}{m})}$ is called the  {\it Hadamard (Schur) geometric mean} of sets $\Psi _1, \ldots , \Psi _m$. If $L\in \mathcal{L}$, $\sum _{i=1} ^m \alpha _i \ge 1$ and if $\Psi _1, \ldots , \Psi _m$ are bounded sets of nonnegative matrices that define operators on $L$, %and if $\alpha _1, \ldots , \alpha _m$ are positive numbers such that
%$\sum _{i=1} ^m \alpha _i \ge 1$, 
then  the set $\Psi _1 ^{( \alpha _1)} \circ \cdots \circ \Psi _m ^{(\alpha _m)}$
%$$\Psi _1 ^{( \alpha _1)} \circ \cdots \circ \Psi _m ^{(\alpha _m)}=\{ A_1 ^{( \alpha _1)} \circ \cdots \circ A _m ^{(\alpha _m)}: A_1\in \Psi _1, \ldots, A_m \in \Psi _m \}$$
is a bounded set of   nonnegative matrices that define operators on $L$ by Theorem \ref{special_case}(ii). 
The following result that follows from Theorem \ref{DBPfs} was established in \cite[Theorem 3.3]{P17}, \cite[Theorems 3.1 and 3.8]{P19} and \cite[Theorem 2.5]{BP22b}.
 %(\cite[Theorem 3.3]{P17}; see also \cite{P12}, \cite{P06}) follows from   Theorem \ref{DPBfs}.

%A version of the following result on the generalized and the joint spectral radius was stated in \cite[Theorem 3.4]{P12} and \cite[Corollary 5.3]{P06} only in the case of bounded sets  of
%non-negative matrices that define operators on Banach sequence spaces,
%however the same proof works in our more general setting by applying the inequalities (\ref{spectral2}) and (\ref{norm2}). The proof is included for the convenience of the reader.

%\begin{theorem}
%Let $\Psi _1, \ldots \Psi _m$ be bounded sets of positive kernel operators on a Banach function space $L$ and let
% $\alpha _1, \ldots \alpha _m$ be positive numbers such that \\
%$\sum _{i=1} ^m \alpha _i = 1$.
%Then we have
%\be
%\rho (\Psi _1 ^{( \alpha _1)} \circ \cdots \circ \Psi _m ^{(\alpha _m)} ) \le
%\rho (\Psi _1)^{ \alpha _1} \, \cdots \rho(\Psi _m)^{\alpha _m}
%\label{gsh}
%\ee
%and
%\be
%\hat{\rho } (\Psi _1 ^{( \alpha _1)} \circ \cdots \circ \Psi _m ^{(\alpha _m)} ) \le
%\hat{\rho }  (\Psi _1)^{ \alpha _1} \, \cdots \hat{\rho } (\Psi _m)^{\alpha _m}.
%\label{gsh2}
%\ee
%\label{family}
%\end{theorem}
\begin{theorem} Let $\Psi _1, \ldots , \Psi _m$ be bounded sets of positive kernel operators on a Banach function space $L$, let
 $\alpha _1, \ldots , \alpha _m$ be positive numbers and $n \in \NN$.
 
\noindent (i) If
$\sum _{i=1} ^m \alpha _i = 1$ and $r \in \{\rho, \hat{\rho}\}$, then
\be
r (\Psi _1 ^{( \alpha _1)} \circ \cdots \circ \Psi _m ^{(\alpha _m)} ) \le  r ((\Psi _1 ^n ) ^{( \alpha _1)} \circ \cdots \circ (\Psi _m ^n) ^{(\alpha _m)} ) ^{ \frac{1}{n}} \le
r(\Psi _1)^{ \alpha _1} \, \cdots r(\Psi _m)^{\alpha _m}
\label{gsh_ref}
\ee
and
\be
r \left(\Psi _1 ^{\left( \frac{1}{m} \right)} \circ \cdots \circ \Psi _m  ^{\left( \frac{1}{m} \right)} \right) \le r(\Psi _1 \Psi _2 \cdots \Psi _m) ^{\frac{1}{m}}.\;\;\;\;\;\;\;\;
\label{Hu_ess}
\ee
If, in addition, $L$ and $L^*$ have order continuous norms, then (\ref{gsh_ref}) and (\ref{Hu_ess}) hold also  for each $r\in \{ \rho _{ess}, \hat{\rho} _{ess}\}$.

\noindent (ii) If $L\in\mathcal L$, $\sum_{j=1}^{m}\alpha_j\ge1$, $r \in \{\rho, \hat{\rho}\}$ and 
if $\Psi, \Psi _1, \ldots , \Psi _m$ are bounded sets of nonnegative matrices that define operators on $L$, 
%. Let $\alpha_1, \ldots , \alpha _m$ be  positive numbers such that $\sum_{j=1}^{m}\alpha_j\ge1$, $n \in \NN$ and $r \in \{\rho, \hat{\rho}\}$. 
then Inequalities (\ref{gsh_ref}) hold.
%\be
%r (\Psi _1 ^{( \alpha _1)} \circ \cdots \circ \Psi _m ^{(\alpha _m)} ) \le  r ((\Psi _1 ^n ) ^{( \alpha _1)} \circ \cdots \circ (\Psi _m ^n) ^{(\alpha _m)} ) ^{ \frac{1}{n}} \le
%r(\Psi _1)^{ \alpha _1} \, \cdots r(\Psi _m)^{\alpha _m}
%\label{wieder}
%\ee

In particular, if $t\ge1$, then
\be
r (\Psi^{(t)})\le r((\Psi^n)^{(t)})^{\frac{1}{n}}\le r(\Psi)^{t}.
\label{folge}
\ee
%\end{theorem}
\label{powers}
\end{theorem}

\section{New inequalities for the Haussdorf measure of noncompactness and essential radius}

In this section we prove that the essential versions of Theorems \ref{thbegin}(ii), \ref{special_case}(ii)-(iii) and \ref{powers}(ii) hold under the assumption that $L$ and $L^*$ have order continuous norms. We will need the following lemma.

\begin{lemma}
Let $L\in \mathcal{L}$ have order continuous norm. Then for each $x\in L$ it holds that $x(i) \to 0$ as $i \to \infty$.
\label{property*}
\end{lemma}
\begin{proof} Suppose there exists $x\in L$ such that the entries $x(i)$  do not converge to zero as $i \to \infty$. Then there exists $\varepsilon > 0$ such that there are infinitely many positive entries of $|x|$ that are greater than $\varepsilon$. For $k  \in \NN$ let $x_k (i)=0$ when $i\le k$ and $x_k(i)=|x|(i)$ otherwise.
Then $0\le x_k \downarrow 0$.  However, $\|x_k\|$ does not converge to zero, since we have
 $\|x_k\| \ge \||x|(i) \cdot e_i\| =|x(i)| >\varepsilon$
  for infinitely many $i>k$.
\end{proof}
First we establish the essential version of Theorem \ref{special_case}(iii).
\begin{theorem}
\label{ess_needed}
Let $L\in \mathcal{L}$ such that %satisfy property (*) and let 
$L$ and $L^*$ have order continuous norms. 
Let $t\ge 1$ and let $A, A_1, \ldots , A_m$ be  nonnegative matrices that define operators on $L$.  %linear operator on $\ell^2(\N)$ and $t \ge 1$.  
Then 
\be
\gamma (A^{(t)}) \le \gamma (A)^t,
\label{newH}
\ee
\be
\rho_{ess} (A^{(t)}) \le \rho_{ess} (A)^t,
\label{new_ess}
\ee
%\be
%\;\;\;\;\;A_1^{(t)}\cdots A_m^{(t)}\le(A_1\cdots A_m)^{(t)},
%\label{gl1t}
%\ee
\be
\gamma (A_1^{(t)}\cdots A_m^{(t)})\le \gamma (A_1\cdots A_m)^{t},
\label{newH2}
%\label{gl1vecrt}
\ee
\be
\rho_{ess}(A_1^{(t)}\cdots A_m^{(t)})\le\rho_{ess} (A_1\cdots A_m)^{t}.
\label{new_ess2}
%\label{gl1nt}
\ee
%$r_{ess}(K^{(t)}) \le r_{ess}(K)^t$ for all $t \ge 1$.  
\end{theorem}

\begin{proof}
First we prove (\ref{newH}). If $\gamma (A)=0$, then $\gamma (A^{(t)})=0$ by (\ref{dobra_gamma}).
We may assume that $t>1$. We may also assume that $\gamma (A)=1$ since $\gamma (\cdot)$ is positively homogeneous.
%It is sufficient to prove that $\beta(K^{(t)}) \le \beta(K)$ for all $t \ge 1$ when $\beta(K) = 1$ since you should be able to normalize $K$. (\hl{details...}) 
 Having $\gamma(A) = 1$ means that for any $\delta > 1$, there is a finite set $U \subset L$ %\ell^2(\N)$ 
such that the image $A(D_L)$ of the closed unit ball $D_L$  is contained in the union $\bigcup_{u \in U} (u+\delta D_L)$.  Since $U$ is a finite set in $L$, then by Lemma \ref{property*} %and since $L$ satisfies property (*), %$\ell^2(\N)$
 there are only finitely many entries $i$ such that $\max_{u \in U} |u_i| > \delta^2-\delta$. Let $I$ denote this set of indices. For all other indices $i \notin I$, we must have $(Ax)_i \le \max |u_i| + \delta \le \delta^2$ for all $x \in D_L$, $x\ge 0$.  In particular, $A_{ij} = (Ae_j)_i \le \delta^2$ for all $j$ and all $i \notin I$.

Then $\delta^{-2t} A_{ij}^t \le A_{ij}$ for all $i \notin I$, $j \in \mathbb{N}$ and $t > 1$.  This means that $\delta^{-2t} A^{(t)} _i \le A _i$ for all rows $A_i$ such that $i \notin I$. %in every row except the rows in $I$. 
Let $P_I$ be the orthogonal projection onto $\mathrm{span} \{e_i : i \in I\}$.  Then $P_I A^{(t)}$ is compact since it has finite dimensional range, and if $Q_I = \mathrm{id} - P_I$, then $\delta^{-2t}Q_I A^{(t)} \le A$ and $\delta^{-2t} \gamma(A^{(t)}) =  \delta^{-2t} \gamma(Q_I A^{(t)}) \le \gamma(A) = 1$ (since $\gamma(\cdot)$ is invariant under compact perturbations and since it is monotone). Then $\gamma(A^{(t)}) \le \delta^{2t}$.  Since $\delta > 1$ can be chosen arbitrarily close to 1, we conclude that $\gamma(A^{(t)}) \le 1$ for all $t > 1$. This proves (\ref{newH}).

Inequality (\ref{newH2}) follows from (\ref{gl1t}), monotonicity of $\gamma(\cdot)$  and   (\ref{newH}). Inequality (\ref{new_ess2}) follows from (\ref{esslim=inf}) and  (\ref{newH2}) since 
$$\rho _{ess} (A) =\lim _{j \to \infty} \gamma ((A_1^{(t)}\cdots A_m^{(t)})^j)^{1/j}\le \lim _{j \to \infty} \gamma ((A_1 \cdots A_m)^j)^{t/j}=\rho_{ess} (A_1\cdots A_m)^{t}.$$
 Inequality  (\ref{new_ess}) is a special case of  (\ref{new_ess2}).
\end{proof}
\begin{remark}{\rm Observe that Theorem \ref{ess_needed} is not a special case of \cite[Lemma 4.2]{P06} since for each $i$ and $j$ we have $\gamma (E_{ij})=0$, where $E_{ij}$ denotes the infinite matrix with $1$ and at the $ij$th coordinate and with $0$ elsewhere.
}
\end{remark}
Applying standard techniques used also in \cite{DP05} and \cite{P06} we establish the essential versions of Theorems \ref{special_case}(ii) and  \ref{thbegin}(ii).% follows %results follow.

%in \cite[Theorem 2.2]{DP05} and
%\cite[Theorem 5.1 and Example 3.7]{P06} (see also e.g. \cite[Theorem 2.1]{P17}).

\begin{theorem}
\label{special_case_ess}
Let $L\in \mathcal{L}$ such that %satisfy property (*) and let 
$L$ and $L^*$ have order continuous norms. 
Assume $A_1, \ldots , A_m$ are  nonnegative matrices that define operators on $L$ and let $\alpha _1, \ldots, \alpha _m$ be positive numbers  such that $s_m= \sum_{j=1}^{m} \alpha _j\ge 1$. Then inequalities
(\ref{gl1meas_nonc}) and (\ref{gl1vecress}) hold.
\end{theorem}
\begin{proof} For $j=1,  \ldots , m$ define $\beta _j = \frac{\alpha _j}{s_m}$ and so $\sum_{j=1}^{m} \beta _j =1$. Then by (\ref{newH}) and Theorem \ref{special_case}(i) we have
$$\gamma (A_1 ^{( \alpha _1)}  \circ \cdots \circ A_m ^{(\alpha _m)} )= \gamma \left( \left(A_1 ^{( \beta _1)} \circ \cdots \circ A_m ^{(\beta _m)} \right)^{(s_m)}\right)
\le  \gamma  \left(A_1 ^{( \beta _1)} \circ  \cdots \circ A_m ^{(\beta _m)} \right)^{s_m}$$
$$\le  \left(\gamma(A_1)^{ \beta _1} \cdots \gamma(A_m)^{\beta _m}\right)^{s_m}= \gamma(A_1)^{ \alpha _1}  \gamma(A_2)^{\alpha _2} \cdots \gamma(A_m)^{\alpha _m},$$
which proves (\ref{gl1meas_nonc}) under our assumptions. Similarly, (\ref{gl1vecress}) follows from (\ref{new_ess}) and  Theorem \ref{special_case}(i) .
\end{proof}
\begin{theorem}
\label{thbegin_ess}
Let $L\in \mathcal{L}$ such that %satisfy property (*) and let 
$L$ and $L^*$ have order continuous norms. 
Assume $\{A_{i j}\}_{i=1, j=1}^{k, m}$ are  nonnegative matrices that define operators on $L$ and let $\alpha _1, \ldots, \alpha _m$ be positive numbers  such that $s_m= \sum_{j=1}^{m} \alpha _j\ge 1$. Then
for $A$ from (\ref{osnovno}) inequalities
(\ref{meas_noncomp}) and (\ref{ess_spectral}) hold.
\end{theorem}
\begin{proof} Inequalities (\ref{meas_noncomp}) and (\ref{ess_spectral}) under our assumptions follow from (\ref{norm2}) in Theorem \ref{DBPfs}(ii), monotonicity of $\gamma (\cdot)$ and $\rho_{ess} (\cdot)$ and from Theorem \ref{special_case_ess}.
\end{proof}

%The following result extends Inequalities (\ref{tri}) and (\ref{gsh_ref}) and Theorem \ref{for_matrices}.
%%

%\begin{theorem} Let $L\in \mathcal{L}$ satisfy property (*) and let $L$ and $L^*$ have order continuous norms. Assume $\alpha _1, \ldots , \alpha _m$ are  positive numbers such that $\sum_{j=1}^{m} \alpha _j\ge 1$ and let $n \in \NN$.  Let  $r\in \{ \rho _{ess}, \hat{\rho} _{ess}\}$  and let  $\Psi_{1}, \ldots , \Psi _m$ be bounded sets of nonnegative matrices that define positive operators on $L$. Then
%\be
%r (\Psi _1 ^{( \alpha _1)} \circ \cdots \circ \Psi _m ^{(\alpha _m)} ) \le  r ((\Psi _1 ^n ) ^{( \alpha _1)} \circ \cdots \circ (\Psi _m ^n) ^{(\alpha _m)} ) ^{ \frac{1}{n}} \le
%r(\Psi _1)^{ \alpha _1} \, \cdots r(\Psi _m)^{\alpha _m}.
%\label{gsh_ref_ess}
%\ee
%\end{theorem}
%\begin{proof}
%...
%\end{proof}
%
%\begin{corollary}
%...
%\end{corollary}

The following result on the joint and generalized essential radius of bounded sets of infinite nonnegative matrices generalizes Theorem \ref{thbegin_ess} and is an essential version of 
 \cite[Theorem 3.3(ii)]{BP22b} (in the case $\sum_{j=1}^{m} \alpha _j= 1$ it is known by  \cite[Theorem 3.3(i)]{BP22b}).  It is proved by combining ideas from the proofs of \cite[Corollary 5.3]{P06}, \cite[Theorem 3.8]{P17+} and \cite[Theorem 3.3]{BP22b}. 

\begin{theorem}
\label{finally_ess}
Let $L\in \mathcal{L}$ such that %satisfy property (*) and let 
$L$ and $L^*$ have order continuous norms. 
 Assume $\alpha _1, \ldots , \alpha _m$ are  positive numbers such that $\sum_{j=1}^{m} \alpha _j\ge 1$ and let $n \in \NN$.  Let  $r\in \{ \rho _{ess}, \hat{\rho} _{ess}\}$  and let $\Psi_{1}, \ldots , \Psi _m$ and
 $\{\Psi_{i j}\}_{i=1, j=1}^{k, m}$ be bounded sets of nonnegative matrices that define positive operators on $L$. Then
\be
r (\Psi _1 ^{( \alpha _1)} \circ \cdots \circ \Psi _m ^{(\alpha _m)} ) \le  r ((\Psi _1 ^n ) ^{( \alpha _1)} \circ \cdots \circ (\Psi _m ^n) ^{(\alpha _m)} ) ^{ \frac{1}{n}} \le
r(\Psi _1)^{ \alpha _1} \, \cdots r(\Psi _m)^{\alpha _m}.
\label{gsh_ref_ess}
\ee
and
\begin{eqnarray}
\nonumber
& & r \left(\left(\Psi_{1 1}^{(\alpha _1)} \circ \cdots \circ \Psi_{1 m}^{(\alpha _m)}\right) \ldots \left(\Psi_{k 1}^{(\alpha _1)} \circ \cdots \circ \Psi_{k m}^{(\alpha _m)} \right) \right) \\
\nonumber
& \le &r \left((\Psi_{1 1} \cdots  \Psi_{k 1})^{(\alpha _1)} \circ \cdots
\circ (\Psi_{1 m} \cdots \Psi_{k m})^{(\alpha _m)}\right)\\
\nonumber
& \le &r \left(((\Psi_{1 1} \cdots  \Psi_{k 1})^n)^{(\alpha _1)} \circ \cdots
\circ ((\Psi_{1 m} \cdots \Psi_{k m})^n)^{(\alpha _m)}\right)^{\frac{1}{n}}\\
&\le&r \left( \Psi_{1 1} \cdots  \Psi_{k 1} \right)^{\alpha _1} \cdots
r \left( \Psi_{1 m} \cdots \Psi_{k m}\right)^{\alpha _m} .
\label{lepa_ess}
\end{eqnarray}
In particular, if $\Psi _1, \ldots, \Psi_k$ are bounded sets of nonnegative matrices that define positive operators on $L$ and
 $t\ge1$, then
\be
r (\Psi _1 ^{(t)} \cdots \Psi _k ^{(t)})\le r ((\Psi _1 \cdots \Psi _k) ^{(t)}) \le r(((\Psi _1 \cdots \Psi _k)^n)^{(t)})^{\frac{1}{n}}\le r(\Psi _1 \cdots \Psi _k)^{t}.
\label{with_t_ess}
\ee 
\end{theorem}
\begin{proof}  
First we prove the inequality
\be
r (\Psi _1 ^{( \alpha _1)} \circ \cdots \circ \Psi _m ^{(\alpha _m)} ) \le 
% r ((\Psi _1 ^n ) ^{( \alpha _1)} \circ \cdots \circ (\Psi _m ^n) ^{(\alpha _m)} ) ^{ \frac{1}{n}} \le
r(\Psi _1)^{ \alpha _1} \, \cdots r(\Psi _m)^{\alpha _m}.
\label{origin}
\ee
Let $A \in (\Psi _1 ^{( \alpha _1)} \circ \cdots \circ \Psi _m ^{(\alpha _m)})^l$, $l \in \NN$. Then there are 
$A_{ik} \in \Psi _k$, $i=1, \ldots ,l$, $k=1,\ldots , m$ such that
$$A=(A_{11} ^{\alpha _1} \circ \cdots \circ A_{1m} ^{\alpha _m}) \cdots (A_{l1} ^{\alpha _1} \circ \cdots \circ A_{lm} ^{\alpha _m}).$$
By  Theorem \ref{thbegin_ess} %(\ref{meas_noncomp}) and (\ref{ess_spectral}) %Remark \ref{mt1} 
we have 
\be
\gamma (A) \le \gamma (A_{11} \cdots A_{l1})^{\alpha _1} \cdots \gamma (A_{1m} \cdots A_{lm})^{\alpha _m},
%\label{useg}
\nonumber
\ee
\be
\rho _{ess}(A) \le \rho _{ess} (A_{11} \cdots A_{l1})^{\alpha _1} \cdots \rho _{ess} (A_{1m} \cdots A_{lm})^{\alpha _m}.
\label{useg}
\nonumber
\ee
Since $A_{1k} \cdots A_{lk} \in \Psi_k ^l$ for all $k=1,\ldots , m$, (\ref{origin}) follows. %(\ref{useg}) 
%implies (\ref{gsh}).

To prove the first inequality in (\ref{lepa_ess}) let $l\in \NN$ and 
$$B \in \left(\left(\Psi_{1 1}^{(\alpha _1)} \circ \cdots \circ \Psi_{1 m}^{(\alpha _m)}\right) \ldots \left(\Psi_{k 1}^{(\alpha _1)} \circ \cdots \circ \Psi_{k m}^{(\alpha _m)} \right)\right)^l.$$
Then $B=A_1 \cdots A_l$, where for each $i=1, \ldots , l$ we have
$$A_i = \left(A_{i1 1}^{(\alpha _1)} \circ \cdots \circ A_{i 1 m}^{(\alpha _m)}\right) \ldots \left(A_{i k 1}^{(\alpha _1)} \circ \cdots \circ A_{i k m}^{(\alpha _m)} \right),$$
where $A_{i1 1} \in \Psi_{1 1}, \ldots , A_{i1 m} \in \Psi_{1 m}, \ldots , A_{ik 1} \in \Psi_{k 1}, \ldots , A_{ikm} \in \Psi_{km}. $ 
Then by (\ref{norm2}) for each $i=1, \ldots , l$ we have
$$A_i \le C_i := (A_{i11}A_{i21} \cdots A_{ik1} )^{(\alpha_1)} \circ \cdots \circ (A_{i1m}A_{i2m} \cdots A_{ikm} )^{(\alpha_m)}, $$
where $C_i \in (\Psi_{1 1} \cdots  \Psi_{k 1})^{(\alpha _1)} \circ \cdots
\circ (\Psi_{1 m} \cdots \Psi_{k m})^{(\alpha _m)}.$ Therefore 
$$B \le C:=C_1 \cdots C_l \in \left((\Psi_{1 1} \cdots  \Psi_{k 1})^{(\alpha _1)} \circ \cdots
\circ (\Psi_{1 m} \cdots \Psi_{k m})^{(\alpha _m)}\right)^l,$$
$\rho_{ess} (B)^{1/l}\le \rho_{ess} (C)^{1/l}$ and $\gamma(B)^{1/l}\le \gamma(C)^{1/l}$, which implies the first inequality in (\ref{lepa_ess}).

The  first inequality in (\ref{gsh_ref_ess}) follows from the  first inequality in (\ref{lepa_ess}) and (\ref{again}), since
$$r (\Psi _1 ^{( \alpha _1)} \circ \cdots \circ \Psi _m ^{(\alpha _m)} ) = r \left( \left(\Psi _1 ^{( \alpha _1)} \circ \cdots \circ \Psi _m ^{(\alpha _m)} \right)^n \right)^{\frac{1}{n}}  $$
$$= r ((\Psi _1 ^{( \alpha _1)} \circ \cdots \circ \Psi _m ^{(\alpha _m)})\cdots (\Psi _1 ^{( \alpha _1)} \circ \cdots \circ \Psi _m ^{(\alpha _m)}))^{\frac{1}{n}} $$
$$\le  r ((\Psi _1 ^n ) ^{( \alpha _1)} \circ \cdots \circ (\Psi _m ^n) ^{(\alpha _m)} ) ^{ \frac{1}{n}}.$$
The second inequality in (\ref{gsh_ref_ess}) follows from (\ref{origin}) and (\ref{again}). 
The second and third inequality in (\ref{lepa_ess}) follow from  (\ref{gsh_ref_ess}). Inequalities 
(\ref{with_t_ess}) are a special case of (\ref{lepa_ess}).
\end{proof}
%We will need the following well-known inequalities (see e.g. \cite{Mi}). For
%nonnegative measurable functions and for nonnegative numbers $\alpha$ and $\beta$  such that $\alpha+\beta \ge 1$ we have
%\be
%f_1^{\alpha}g_1^{\beta}+\cdots+f_m^{\alpha}g_m^{\beta} \le (f_1+\cdots+f_m)^{\alpha}(g_1+\cdots+g_m)^{\beta}
%\label{mitrn} %{good_corollary}
%\ee

\begin{remark}{\rm Under the assumptions of Theorem \ref{thbegin_ess} an analogue of \cite[Theorem 3.4]{BP21} is valid. The proof runs by following the same lines as in the proof of this result. The details are omitted.
}
\end{remark}

\section{Further results}

By applying results of the previous section we obtain several results that are essential versions of known relatively recent results from the literature. Since the proofs are similar to the existing proofs we mostly omit them  to avoid too much repetition of ideas.

Recall that for nonnegative measurable functions $\{f _{i j}\}_{i=1, j=1}^{k, m}$ and for nonnegative numbers  $\alpha_j$, $j=1, \ldots , m$,   such that $\sum _{j=1} ^m \alpha _j \ge 1$ (see e.g. \cite{Mi}, \cite{BP22b}) we have 
\be
(f_{11}^{\alpha _1} \cdots f_{1m}^{\alpha _m})+\cdots+(f_{k1}^{\alpha _1} \cdots f_{km}^{\alpha _m}) \le (f_{11}+\cdots+f_{k1})^{\alpha _1} \cdots (f_{1m}+\cdots+f_{km})^{\alpha _m}.
\label{mitr2}
\ee
The sum of bounded sets $\Psi$ and $\Sigma$ is a bounded set defined by
$\Psi + \Sigma =\{A+B: A \in \Psi, B \in \Sigma\}$. The following result is an essential version of 
 \cite[Theorem 3.7(ii)]{BP22b} (in the case $\sum_{j=1}^{m} \alpha _j= 1$ it is known by  \cite[Theorem 3.7(i)]{BP22b}).  It is proved in a similar way as  \cite[Theorem 3.7]{BP22b} (it follows from (\ref{mitr2}) and (\ref{gsh_ref_ess})). %and the proof is omitted to avoid too much repetition of ideas.
\begin{theorem}
\label{finally2_ess}
Let $L\in \mathcal{L}$ such that %satisfy property (*) and let 
$L$ and $L^*$ have order continuous norms. 
 Assume $\alpha _1, \ldots , \alpha _m$ are  positive numbers such that $\sum_{j=1}^{m} \alpha _j\ge 1$ and let $n \in \NN$.  Let  $r\in \{ \rho _{ess}, \hat{\rho} _{ess}\}$  and let  $\{\Psi_{i j}\}_{i=1, j=1}^{k, m}$ be bounded sets of nonnegative matrices that define positive operators on $L$. Then
\begin{eqnarray}
\nonumber
& & r \left(\left(\Psi_{1 1}^{(\alpha _1)} \circ \cdots \circ \Psi_{1 m}^{(\alpha _m)}\right) + \ldots + \left(\Psi_{k 1}^{(\alpha _1)} \circ \cdots \circ \Psi_{k m}^{(\alpha _m)} \right) \right) \\
\nonumber
& \le &r \left((\Psi_{1 1}+ \cdots  + \Psi_{k 1})^{(\alpha _1)} \circ \cdots
\circ (\Psi_{1 m} + \cdots  + \Psi_{k m})^{(\alpha _m)}\right)\\
\nonumber
& \le &r \left(((\Psi_{1 1} + \cdots  + \Psi_{k 1})^n)^{(\alpha _1)} \circ \cdots
\circ ((\Psi_{1 m} + \cdots + \Psi_{k m})^n)^{(\alpha _m)}\right)^{\frac{1}{n}}\\
&\le&r \left( \Psi_{1 1} + \cdots  + \Psi_{k 1} \right)^{\alpha _1} \cdots
r \left( \Psi_{1 m} + \cdots  + \Psi_{k m}\right)^{\alpha _m} .
\label{lepa2_ess}
\end{eqnarray}
\end{theorem}
Next we turn our attention to the weighted geometic symmetrizations of sets of infinite matrices (see \cite{BP22b}). 
Let $\Psi$ be a bounded set of nonnegative matrices that define operators on $l^2$ and denote $\Psi ^* =\{A^* : A \in \Psi\}$. Observe that $(\Psi \Sigma)^*= \Sigma ^* \Psi ^*$, $(\Psi ^m)^* = (\Psi ^*)^m$ for all $m \in \NN$ and $r(\Psi)=r(\Psi^*)$ for all $r\in \{ \rho , \hat{\rho} , \rho _{ess}, \hat{\rho} _{ess}\}$. Let $\alpha$ and $\beta$ be nonnegative numbers such that $\alpha+\beta \ge 1$. 
The weighted geometric symmetrization set $S_{\alpha, \beta}(\Psi)=\Psi^{(\alpha)}\circ(\Psi^*)^{(\beta)}=\{A^{(\alpha)}\circ (B^*)^{(\beta)} : A, B \in \Psi\}$ is a bounded set of nonnegative matrices that define operators on $l^2$ by Theorem \ref{thbegin}(ii). 

The following two results are essential versions of \cite[Proposition 4.4 and Theorem 4.3]{BP22b}. They follow from Theorems \ref{finally_ess} and \ref{finally2_ess} and are proved in a very similar way as \cite[Proposition 4.4 and Theorem 4.3]{BP22b}.
\begin{proposition}
Let $\Psi$, $\Psi_1, \ldots ,\Psi_m$ be bounded sets of nonnegative matrices that define operators on $l^2$, $n \in \mathbb{N}$ and let $\alpha$ and $\beta$ be nonnegative numbers such that $\alpha+\beta \ge 1$. Then we have
\begin{eqnarray}
\nonumber
& & r(S_{\alpha ,\beta} (\Psi_{1}) \cdots  S_{\alpha  ,\beta}(\Psi_{m})) 
\le  r\left((\Psi_1 \cdots \Psi_m )^{(\alpha)} \circ ((\Psi_m \cdots \Psi_1)^*)^{(\beta)}  \right) \\
\nonumber
&\le & r\left(((\Psi_1 \cdots \Psi_m )^n)^{(\alpha)} \circ (((\Psi_m \cdots \Psi_1)^*)^n)^{(\beta)}\right)^{\frac{1}{n}}\\
&\le &  r(\Psi_1 \cdots \Psi_m )^{\alpha} \,  r(\Psi_m \cdots \Psi_1 )^{\beta}, 
\label{geom_sym_prva} \\
& & r(S_{\alpha, \beta}(\Psi)) \le r(S_{\alpha, \beta}(\Psi ^n))^{\frac{1}{n}}\le r(\Psi)^{\alpha+\beta},
\label{geom_sym_druga}
\end{eqnarray}
%\be r(S_{\alpha, \beta}(\Psi)) \le r(S_{\alpha, \beta}(\Psi ^n))^{\frac{1}{n}}\le r(\Psi)^{\alpha+\beta},
%\label{geom_sym_druga}
%\ee
\begin{eqnarray}
\nonumber
& & r(S_{\alpha  ,\beta}(\Psi_1)+ \cdots + S_{\alpha  ,\beta}(\Psi_m)) \le r\left(S_{\alpha  ,\beta}(\Psi_1+ \cdots+ \Psi_m)\right) \\
&\le & r\left(S_{\alpha  ,\beta}((\Psi_1+ \cdots+ \Psi_m)^{n})\right)^{\frac{1}{n}} \le r(\Psi_1+ \cdots+ \Psi_m)^{\alpha + \beta},
\label{geom_sym_treca}
\end{eqnarray}
\begin{eqnarray}
\nonumber
& & r(S_{\alpha ,\beta}(\Psi_1)S_{\alpha ,\beta}(\Psi_2))\le r\left((\Psi_1 \Psi_2 )^{(\alpha)} \circ ((\Psi_2 \Psi_1)^*)^{(\beta)}  \right) \\
&\le & r\left(((\Psi_1 \Psi_2 )^n)^{(\alpha)} \circ (((\Psi_2 \Psi_1)^*)^n)^{(\beta)}\right)^{\frac{1}{n}}
\le r(\Psi_1\Psi_2)^{\alpha+\beta}
\label{geom_sym_cetvrta}
\end{eqnarray}
for all $r \in \{ \rho _{ess}, \hat{\rho}_{ess}\}$.
\end{proposition}
\begin{theorem}
\label{sym_matrices}
Let $\Psi$ be a bounded set of nonnegative matrices that define operators on $l^2$ and $r \in \{ \rho _{ess}, \hat{\rho} _{ess}\}$. Assume $\alpha$ and $\beta$  are nonnegative numbers such that $\alpha+\beta \ge 1$ and denote $r_n=r(S_{\alpha,\beta}(\Psi^{2^n}))^{2^{-n}}$ for $n \in \mathbb{N}\cup \{0\}$. Then we have
\be
r(S_{\alpha, \beta}(\Psi))=r_{0}\le r_1\le \cdots \le r_n\le r(\Psi)^{\alpha+\beta}.
\label{finish}
\ee
\end{theorem}
\begin{remark}{\rm Theorem \ref{sym_matrices} implies  that also the essential version of \cite[Theorem 2.5(ii)]{BP21} is valid.
}
\end{remark}
The following result is an essential version of \cite[Theorem 3.1(ii)]{BP22b} and is proved in a similar way as this result by applying Theorem \ref{finally_ess}.  
\begin{proposition}
Let $L\in \mathcal{L}$ such that %satisfy property (*) and let 
$L$ and $L^*$ have order continuous norms.  Assume  $r\in \{ \rho _{ess}, \hat{\rho} _{ess}\}$, $ m,n \in\NN$, $\alpha\ge 1$ and let $\Psi$ be %and $\Sigma$ be
 a  bounded set of nonnegative matrices that define operators on $L$. Then
\be
r(\Psi^{(m)})\le r(\Psi\circ\cdots\circ\Psi)\le r(\Psi ^n \circ\cdots\circ\Psi^n)^{\frac{1}{n}}\le r(\Psi)^{m},
\label{tre}
\ee
where in (\ref{tre}) the Hadamard products in $\Psi\circ\cdots\circ\Psi$ and in $\Psi ^n\circ\cdots\circ\Psi ^n$ are taken $m-1$ times, and
\be
r(\Psi^{(\alpha)})\le r(\Psi^{(\alpha-1)}\circ\Psi)\le  r((\Psi ^n)^{(\alpha-1)}\circ\Psi ^n)^{\frac{1}{n}} \le r(\Psi)^{\alpha}.
\label{tre1}
\ee
\end{proposition}
The following result is an essential version of \cite[Theorem 3.6]{BP22b} and is proved in a similar way as this result by applying Theorems \ref{finally_ess}, \ref{powers}, property (\ref{again}) and \cite[Theorem 3.4]{BP22b}.
\begin{theorem} 
\label{kathyth2_ess}
Let $L\in \mathcal{L}$ such that %satisfy property (*) and let 
$L$ and $L^*$ have order continuous norms. 
Let $\Psi _1, \ldots , \Psi _m$ be bounded sets of nonnegative matrices that define operators on L and
 $ \Phi _j=\Psi_j\ldots\Psi_m \Psi_1\ldots\Psi_{j-1}$ for $j=1,\ldots , m$. Assume that
 $\alpha\ge\frac{1}{m}$, $\alpha _j \ge 0$, $j=1, \ldots , m$, $\sum_{j=1}^{m}\alpha_j \ge 1$ and $n \in \NN$. If $r\in \{ \rho _{ess}, \hat{\rho} _{ess}\}$  and $ \Sigma _j=\Psi_j ^{(\alpha m)}\ldots\Psi_m ^{(\alpha m)} \Psi_1 ^{(\alpha m)} \ldots\Psi_{j-1} ^{(\alpha m)}$ for $j=1,\ldots , m$, then we have
\begin{eqnarray}
\nonumber
& & r\left(\Psi _1^{(\alpha)}\circ\cdots\circ\Psi _m^{(\alpha)}\right)\le r\left(\Phi _1^{(\alpha)}\circ\cdots\circ\Phi _m^{(\alpha)}\right)^{\frac{1}{m}} \\
&\le& r\left((\Phi _1 ^n)^{(\alpha)}\circ\cdots\circ(\Phi _m ^n)^{(\alpha)}\right)^{\frac{1}{mn}} \le r\left(\Psi_1\cdots\Psi_m\right)^{\alpha},
\label{ineqx} 
\end{eqnarray}
\begin{eqnarray}
\nonumber
& & r\left(\Psi _1^{(\alpha)}\circ\cdots\circ\Psi _m^{(\alpha)}\right)\le r\left(\Psi_1^{(\alpha m)}\cdots\Psi_m^{(\alpha m)}\right)^{\frac{1}{m}}\\
&\le&  r\left((\Psi_1\cdots\Psi_m)^{(\alpha m)}\right)^{\frac{1}{m}}\le  r\left(((\Psi_1\cdots\Psi_m)^n)^{(\alpha m)}\right)^{\frac{1}{nm}} \le
r\left(\Psi_1\cdots\Psi_m\right)^{\alpha}. \;\;\;\;\;\;
\label{ineqx2}
\end{eqnarray}
If, in addition, $\alpha \ge 1$ then % and $1 \le t \le m$ then
\begin{eqnarray}
\nonumber
& & r\left(\Psi_1^{(\alpha)}\circ\cdots\circ\Psi_m^{(\alpha)}\right)  \le  r\left(\Phi_1^{(\alpha)}\circ\cdots\circ
\Phi_m^{(\alpha)}\right)^{\frac{1}{m}} \le  r\left((\Phi _1 ^n)^{(\alpha)}\circ\cdots\circ(\Phi _m ^n)^{(\alpha)}\right)^{\frac{1}{mn}} \\
&\le &\left( r\left((\Phi_1 ^n)^{(m)}\right)\cdots r\left((\Phi_m ^n)^{(m)}\right)\right)^{\frac{\alpha}{m^2 n}}
 \le r\left(\Psi_1\cdots\Psi_m\right)^{\alpha},
\label{xyz}
\end{eqnarray}
\begin{eqnarray}
\nonumber
& & r\left(\Psi _1^{(\alpha)}\circ\cdots\circ\Psi _m^{(\alpha)}\right)\le  r\left(\Sigma _1^{(\frac{1}{m})}\circ\cdots\circ\Sigma _m^{(\frac{1}{m})}\right)^{\frac{1}{m}} \\
\nonumber
&\le &  r\left((\Sigma _1 ^n)^{(\frac{1}{m})}\circ\cdots\circ(\Sigma _m ^n)^{(\frac{1}{m})}\right)^{\frac{1}{mn}}
 \le r\left(\Psi_1^{(\alpha m)}\cdots\Psi_m^{(\alpha m)}\right)^{\frac{1}{m}} \\
&\le &  r\left((\Psi_1\cdots\Psi_m)^{(\alpha m)}\right)^{\frac{1}{m}}\le  r\left(((\Psi_1\cdots\Psi_m)^n)^{(\alpha m)}\right)^{\frac{1}{nm}} \le r\left(\Psi_1\cdots\Psi_m\right)^{\alpha}. \;\;\;\;\;\;
\label{kraj}
\end{eqnarray}
\end{theorem}
The following consequence provides the essential version of some of  the main results of \cite{P12} (\cite[Theorems 3.5 and 3.7]{P12}).
\begin{corollary}Let $L\in \mathcal{L}$ such that %satisfy property (*) and let 
$L$ and $L^*$ have order continuous norms. 
Let $\Psi _1$ and $\Psi _2$ be bounded sets of nonnegative matrices that define operators on L, let $r\in \{ \rho _{ess}, \hat{\rho} _{ess}\}$ and $\beta \in [0,1]$. Then  
$$
r (\Psi _1 \circ \Psi_2)  \le r (\Psi_1 ^{(2)} \Psi _2 ^{(2)})^{\frac{1}{2}} \le r ((\Psi_1 \circ \Psi_1)(\Psi _2 \circ \Psi _2))^{\frac{1}{2}}  \le  $$
\be
\le r(\Psi _1 \Psi _2\circ \Psi_1 \Psi_2)^{\frac{\beta}{2}} r(\Psi_2 \Psi_1 \circ \Psi_2\Psi_1)^{\frac{1-\beta}{2}} \le r(\Psi_1\Psi_2)
\label{P1_ess_j}
\ee
and
$$
r (\Psi_1 \circ \Psi_2) \le r (\Psi_1\Psi_2\circ \Psi_2\Psi_1)^{\frac{1}{2}} \le r((\Psi_1\Psi_2)^{(2)})^{\frac{1}{4}}r((\Psi_2\Psi_1)^{(2)})^{\frac{1}{4}} $$
\be
\le  r(\Psi_1\Psi_2\circ \Psi_1\Psi_2)^{\frac{1}{4}} r(\Psi_2\Psi_1 \circ \Psi_2\Psi_1)^{\frac{1}{4}} \le  r (\Psi_1\Psi_2).
\label{P2_ess_j}
\ee
\label{P12_cor}
\end{corollary}
\begin{proof} The first inequality in (\ref{P1_ess_j}) is a special case of the first inequality in (\ref{ineqx2}). The second inequality (\ref{ineqx2}) is trivial, since $\Psi_i ^{(2)} \subset \Psi _i \circ \Psi_i$ for $i=1,2$.  The third inequality in (\ref{P1_ess_j}) follows from the first inequality in (\ref{lepa_ess}) and from (\ref{again}), while the fourt inequality in (\ref{P1_ess_j}) follows from (\ref{gsh_ref_ess}) and  (\ref{again}).

The first inequality in (\ref{P2_ess_j}) is a special case of the first inequality in (\ref{ineqx}). To prove the second and third inequality in (\ref{P2_ess_j}) observe that
$$\Psi_1\Psi_2\circ \Psi_2\Psi_1 = ( (\Psi_1\Psi_2)^{(2)})^{(\frac{1}{2})}\circ ((\Psi_2\Psi_1)^{(2)})^{(\frac{1}{2})}$$
It folllows from  (\ref{gsh_ref_ess})  and (\ref{tre})  that 
$$ r (\Psi_1\Psi_2\circ \Psi_2\Psi_1) \le r((\Psi_1\Psi_2)^{(2)})^{\frac{1}{2}}r((\Psi_2\Psi_1)^{(2)})^{\frac{1}{2}} \le  r(\Psi_1\Psi_2\circ \Psi_1\Psi_2)^{\frac{1}{2}} r(\Psi_2\Psi_1 \circ \Psi_2\Psi_1)^{\frac{1}{2}}, $$
which establishes the second and third inequality in (\ref{P2_ess_j}). The fourth inequality in (\ref{P2_ess_j}) follows from (\ref{gsh_ref_ess}) and  (\ref{again}), which completes the proof.
\end{proof}
%In a special case of singelton sets $\Psi _1=\{A\}$ and $\Psi _1=\{B\}$ we obtain the essential versions of (\ref{P1}) and (\ref{P2}). 
\begin{remark} {\rm Under the assumptions of Theorem \ref{P12_cor} one can similarly as (\ref{P2_ess_j}) prove its variant:
\be
r (\Psi_1 \circ \Psi_2) \le r (\Psi_1\Psi_2\circ \Psi_2\Psi_1)^{\frac{1}{2}} \le r((\Psi_1\Psi_2)^{(\frac{1}{\beta})})^{\frac{\beta}{2}}r((\Psi_2\Psi_1)^{(\frac{1}{1-\beta})})^{\frac{1-\beta}{2}} 
\le  r (\Psi_1\Psi_2).
\label{P3_ess_j}
\ee
Similarly (\ref{P3_ess_j}) is proved if  $L\in \mathcal{L}$ and $r\in \{ \rho, \hat{\rho}\}$.
}
\end{remark}
In a special case of singelton sets $\Psi _1=\{A\}$ and $\Psi _1=\{B\}$ we obtain the essential versions of (\ref{P1}) and (\ref{P2}) (infact a slight generalization). 
\begin{corollary}Let $L\in \mathcal{L}$ such that %satisfy property (*) and let 
$L$ and $L^*$ have order continuous norms. 
Let $A$ and $B$ be nonnegative matrices that define operators on L and let $\beta \in [0,1]$. Then  
\be
\rho _{ess} (A\circ B) \le \rho_{ess} ((A\circ A)(B\circ B))^{\frac{1}{2}}\le \rho_{ess}(AB \circ AB)^{\frac{\beta}{2}} \rho_{ess}(BA \circ BA)^{\frac{1-\beta}{2}} \le \rho_{ess} (AB)
\label{P1_ess}
\ee
and
\be
\rho_{ess} (A\circ B) \le \rho_{ess} (AB\circ BA)^{\frac{1}{2}}\le  \rho_{ess}((AB)^{(\frac{1}{\beta})}))^{\frac{\beta}{2}} \rho_{ess}((BA)^{(\frac{1}{1-\beta})})^{\frac{1-\beta}{2}}  \le  \rho_{ess} (AB).
\label{P2_ess}
\ee
\label{essential_Aud}
\end{corollary}

The following results is an essential version of \cite[Lemma 3.16]{B23+} and is proved in a similar way as this result by applying Theorem \ref{finally_ess}.

 %generalization of \cite[Lemma 3.13]{BP21} from which can be derived generalization of \cite[Corollary 3.14]{BP21}.
\begin{proposition}
\label{dog}
Let $\alpha\ge\frac{1}{2}$, $r\in \{ \rho _{ess}, \hat{\rho} _{ess}\}$ and let $\Psi$ be bounded set of nonnegative matrices that define operators on $l^2$.Then
\be
r(\Psi^{(\alpha)}\circ(\Psi^*)^{(\alpha)})\le r(\Psi^{(\alpha)}\circ\Psi^{(\alpha)})\le r(\Psi)^{2\alpha}
\label{hund}
\ee
\end{proposition}
The following special case is an essential version of \cite[Lemma 3.13]{BP21}.
%\begin{proof}
%To prove the first inequality in \eqref{hund} observe that
% $\Psi^{(\alpha)}\circ(\Psi^*)^{(\alpha)}\subset (\Psi^{(\alpha)}\circ\Psi^{(\alpha)})^{(\frac{1}{2})}\circ((\Psi^*)^{(\alpha)}\circ(\Psi^*)^{(\alpha)})^{(\frac{1}{2})}$ since $A^{(\alpha)}\circ(B^*)^{(\alpha)}=(A^{(\alpha)}\circ A^{(\alpha)})^{(\frac{1}{2})}\circ ((B^*)^{(\alpha)}\circ(B^*)^{(\alpha)})^{(\frac{1}{2})}$ for all $A, B\in\Psi$. It follows that $r(\Psi^{(\alpha)}\circ(\Psi^*)^{(\alpha)})\le r((\Psi^{(\alpha)}\circ\Psi^{(\alpha)})^{(\frac{1}{2})}\circ((\Psi^*)^{(\alpha)}\circ(\Psi^*)^{(\alpha)})^{(\frac{1}{2})})\le r(\Psi^{(\alpha)}\circ\Psi^{(\alpha)})^{\frac{1}{2}}r((\Psi^*)^{(\alpha)}\circ(\Psi^*)^{(\alpha)})^{\frac{1}{2}}=r(\Psi^{(\alpha)}\circ\Psi^{(\alpha)})\le r(\Psi)^{2\alpha}$ by \cite[Theorem 3.3(ii)]{BP21}. 
%\end{proof}
\begin{corollary}
\label{dog_cor}
Let $\alpha\ge\frac{1}{2}$ and  %$r\in \{ \rho _{ess}, \hat{\rho} _{ess}\}$ 
let $A$ be a nonnegative matrix that defines an operators on $l^2$.Then
\be
\rho _{ess}(A^{(\alpha)}\circ(A^*)^{(\alpha)})\le \rho _{ess}(A^{(\alpha)}\circ A^{(\alpha)})\le \rho _{ess}(A)^{2\alpha}
\label{hund_cor}
\ee
\end{corollary}
\bigskip

\section{Further results on $L^2(X,\mu)$}

In this section we will apply a fact that for a bounded linear operator $T$  defined on a Hilbert space we have
\begin{equation}
\rho_{ess}(T^* T) = \rho_{ess}(TT^*)=\gamma (T^* T) = \gamma (TT^*)=  \gamma(T)^2.
\label{Hilbert}
\end{equation}
Since we do not know if this result has previously been known or not, we prove it below in Lemma \ref{Hilb1}.

 Recall that a bounded linear operator on a Hilbert space $\mathcal{H}$ is \emph{hyponormal} if $\|Tx \| \ge \|T^*x\|$ for all $x \in \mathcal{H}$, or equivalently if $T^*T - TT^*$ is positive semidefinite. In particular, any normal operator is hyponormal.  Let $B(\mathcal{H})$ denote the Banach algebra of bounded linear operators on $\mathcal{H}$  and let $\pi$ be the canonical projection  of $B(\mathcal{H})$ onto the Calkin algebra $B(\mathcal{H})/K(\mathcal{H})$. Since the set of $K(\mathcal{H})$ compact  operators in $B(\mathcal{H})$  is a closed two-sided ideal in $B(\mathcal{H})$, the Calkin algebra is a $C^*$-algebra and the canonical projection is a $*$-isomorphism.  The essential norm of $T \in B(\mathcal{H})$ is by definition $\|T\|_{ess} = \|\pi(T)\|$ and we have  $\rho_{ess}(T) = \rho(\pi(T))$.  
 
 The following proposition is probably known as it combines well-known results of Nussbaum and Stampfli \cite{Nussbaum70, Stampfli62}, but we are unaware of a direct reference.
\begin{proposition} \label{prop:NussbaumStampfli}
Let $\mathcal{H}$ be a Hilbert space. % and let $B(\mathcal{H})$ denote the bounded linear operators on $\mathcal{H}$.  
If $T \in B(\mathcal{H})$ is hyponormal, then $$\rho_\text{ess}(T) = \gamma(T) = \|T\|_\text{ess}.$$
\end{proposition}
\begin{proof} 
For any $T \in B(\mathcal{H})$ and $K \in K(\mathcal{H})$ % where $K(\mathcal{H})$ denotes the compact operators in $B(\mathcal{H})$, 
it is clear that $\gamma(T) = \gamma(T+K) \le \|T+K\|$.  Therefore $\gamma(T) \le \|T\|_{ess}$.  By (\ref{esslim=inf}) (\cite[Theorem 1]{Nussbaum70})
%$$\rho_{ess}(T) = \lim_{n \rightarrow \infty} \gamma(T^n)^{1/n} = \lim_{n \rightarrow \infty} \|T^n\|_{ess}^{1/n}.$$ 
%Furthermore, 
and since $\gamma(T^n) \le \gamma(T)^n$ for all $n$, it follows that
$$\rho_{ess}(T) \le \gamma(T) \le \|T\|_{ess}.$$
%for all $T \in B(\mathcal{H})$.

It remains to show that $\rho_{ess}(T) = \|T\|_{ess}$ when $T$ is hyponormal.  
%Let $\pi$ be the canonical projection of $B(\mathcal{H})$ onto the Calkin algebra $B(\mathcal{H})/K(\mathcal{H})$. Since $K(\mathcal{H})$ is a closed two-sided ideal in $B(\mathcal{H})$, the Calkin algebra is a $C^*$-algebra and the canonical projection is a $*$-isomorphism.  The essential norm of $T \in B(\mathcal{H})$ is by definition $\|T\|_{ess} = \|\pi(T)\|$. Similarly, $\rho_{ess}(T) = \rho(\pi(T))$.  
Since the spectrum of $\pi(T^*T-TT^*)$ is a subset of the spectrum of $T^*T-TT^*$ (see e.g., \cite[Theorem 2.3]{FillmoreStampfliWilliams72}), it follows that $\pi(T^*T-TT^*)$ is positive and therefore $\pi(T)$ is hyponormal whenever $T$ is hyponormal. In that case, \cite[Theorem 1]{Stampfli62} says that $\rho(\pi(T)) = \|\pi(T)\|$ and therefore $\rho_{ess}(T) = \|T\|_{ess}$. 
\end{proof}

\begin{lemma}
\label{Hilb1}
Let $\mathcal{H}$ be a Hilbert space and $T \in B(\mathcal{H})$. % and let $B(\mathcal{H})$ denote the bounded linear operators on $\mathcal{H}$.  If $T \in B(\mathcal{H})$, 
Then $\rho_{ess}(T^*T) =\gamma(T^*T)= \gamma(T)^2$.  Consequently, Equalities (\ref{Hilbert}) and $\gamma (T)=\gamma (T^*)$ hold.
\end{lemma}

\begin{proof}
By the polar decomposition theorem for bounded operators on a Hilbert space, $T = UN$ where $U$ is a partial isometry and $N = \sqrt{T^*T}$. It follows immediately that $\rho_{ess}(T^*T) = \rho_{ess}(N^2) = \rho_{ess}(N)^2$.   

By Proposition \ref{prop:NussbaumStampfli}, $\rho_{ess} (N) = \gamma(N)$. Since $U$ is a partial isometry, $\gamma(U) \le \|U\| \le 1$.  So we have:
%Finally we have $U$ is a partial isometry if $U^*U$ is a projection onto the orthogonal complement of the kernel of $U$.  Then any $x \in H$ can be decomposed into $y + z$ with $y,z$ orthogonal such that $Ux = Uy$. So 
%$$\inner{Ux,Ux} = \inner{x,U^*Ux} = \inner{x,y} = \inner{y,y} \le \inner{x,x}$$
%So $\|U\| \le 1$ and since $\gamma(U) \le \|U\|$, it follows that $\gamma(U) \le 1$.  
\be
\gamma(T)^2 = \gamma(UN)^2 \le \gamma(N)^2 = \rho_{ess}(N)^2 = \rho_{ess}(T^*T).
\label{one}
\ee
It remains to prove the reverse inequality.  Since $\gamma(U^*) \le \|U^*\| =\|U\| \le 1$, we have
$$\gamma(T^*T) = \gamma(N U^* T) \le \gamma(N) \gamma(T).$$
Since $\rho_{ess}(T^*T) = \gamma(T^*T) = \gamma(N)^2$,
we conclude that $\rho_{ess}(T^*T) \le \gamma(T)^2$, which together with (\ref{one}) establishes
$\rho_{ess}(T^*T) =\gamma(T^*T)= \gamma(T)^2$. By (\ref{again}) and Proposition \ref{prop:NussbaumStampfli} also the remaining equalities in (\ref{Hilbert}) follow. The equality $\gamma (T)=\gamma (T^*)$ follows from (\ref{Hilbert}).
\end{proof}

Let $\Sigma$ be a bounded set of bounded operators on a Hilbert space $\mathcal{H}$ and let us denote
$$\gamma (\Sigma) = \sup _{T \in \Sigma } \gamma (T) \;\; \mathrm{and}\;\; \|\Sigma\| = \sup _{T \in \Sigma } \|T\|. $$
By $\Sigma^*$  we denote a bounded set of bounded operators on $\mathcal{H}$ defined by $\Sigma ^*=\{T^*: T \in\Sigma \}.$ The following lemma is an essential version of \cite[Lemma 3.1.]{B23+} and it also slightly generalizes it (with a similar proof).

\begin{lemma} Let $\mathcal{H}$ be a Hilbert space and $\Sigma \subset B(\mathcal{H})$ be a bounded set. Then 
\be
\gamma (\Sigma )=\rho _{ess}(\Sigma^*\Sigma)^{1/2}=\rho _{ess}(\Sigma\Sigma^*)^{1/2}=\hat{\rho} _{ess}(\Sigma^*\Sigma)^{1/2}=\hat{\rho} _{ess} (\Sigma \Sigma^*)^{1/2},
\label{tool_ess}
\ee
$\gamma (\Sigma^*)= \gamma (\Sigma)$ 
and
\be
\|\Sigma\|=\rho(\Sigma^*\Sigma)^{1/2}=\rho(\Sigma\Sigma^*)^{1/2}=\hat{\rho}(\Sigma^*\Sigma)^{1/2}=\hat{\rho}(\Sigma \Sigma^*)^{1/2}.
\label{tool}
\ee
\end{lemma}

\begin{proof}
First we prove (\ref{tool_ess}). By Lemma \ref{Hilb1} we have
\begin{eqnarray}
\nonumber
& & \gamma (\Sigma)=\sup_{T\in \Sigma} \gamma (T)=\sup_{T\in \Sigma}\rho _{ess}(T^*T)^{\frac{1}{2}}=(\sup_{T \in \Sigma}\rho _{ess} ((T^*T)^m)^{\frac{1}{m}})^{\frac{1}{2}} \\
 &\le & \left (\sup _{m\in \NN}\sup_{S\in( \Sigma^* \Sigma)^m}\rho_{ess} (S)^{\frac{1}{m}}\right)^{\frac{1}{2}} 
\nonumber
=\rho  _{ess} ( \Sigma^* \Sigma)^{\frac{1}{2}}\le\hat{\rho}_{ess}( \Sigma^* \Sigma)^{\frac{1}{2}}\le \gamma ( \Sigma^* \Sigma )^{\frac{1}{2}}\\
&\le&(\gamma ( \Sigma^*) \gamma ( \Sigma))^{\frac{1}{2}}= \gamma ( \Sigma),
\nonumber
\end{eqnarray} 
which  proves  $\gamma (\Sigma )=\rho _{ess}(\Sigma^*\Sigma)^{1/2}=\hat{\rho} _{ess}(\Sigma^*\Sigma)^{1/2}$. Other equalities in (\ref{tool_ess}) follow again by Lemma \ref{Hilb1}. Equality $\gamma (\Sigma^*)= \gamma (\Sigma)$ follows from  (\ref{tool_ess}) (or also from Lemma \ref{Hilb1}). 

Equalities (\ref{tool}) are proved similarly.
\end{proof}
By applying (\ref{tool_ess}) we obtain the following result, which is an essential version of 
\cite[Theorem 3.2]{B23+} and is proved in a  similar way. For the sake of clarity we include the proof.

\begin{theorem}
\label{first_ess}
Let $\Psi_1, \ldots ,\Psi_m$ be bounded sets of positive kernel operators on $L^2(X, \mu)$  and let  $r \in\{\rho _{ess},\hat{\rho} _{ess}\}$.

If $m$ is even, then
\begin{align}
\nonumber
\gamma (\Psi_1^{(\frac{1}{m})}\circ\cdots\circ\Psi_m^{(\frac{1}{m})}) \le (r(\Psi_1^*\Psi_2\Psi_3^*\Psi_4\cdots\Psi_{m-1}^*\Psi_m)r(\Psi_1\Psi_2^*\Psi_3\Psi_4^*\cdots\Psi_{m-1}\Psi_m^*))^{\frac{1}{2m}}\\
%\nonumber
=(r(\Psi_1^*\Psi_2\Psi_3^*\Psi_4\cdots\Psi_{m-1}^*\Psi_m)r(\Psi_m\Psi_{m-1}^*\cdots\Psi_4\Psi_3^*\Psi_2\Psi_1^*))^{\frac{1}{2m}}\;\;\;\;\;\;\;\;\;\;\;\;\;\;\;\;\;\;\;\;\;\;\;\;\;\;\;\;\;\;\;
\label{ess_ineq1}
\end{align}

If $m$ is odd, then
\begin{align}
\nonumber
\gamma (\Psi_1^{(\frac{1}{m})}\circ\cdots\circ\Psi_m^{(\frac{1}{m})})\;\;\;\;\;\;\;\;\;\;\;\;\;\;\;\;\;\;\;\;\;\;\;\;\;\;\;\;\;\;\;\;\;\;\;\;\;\;\;\;\;\;\;\;\;\;\;\;\;\;\;\;\;\;\;\;\;\;\;\;\;\;\;\;\;\;\;\;\;\;\;\;\;\;\;\;\;\;\;\;\;\;\;\;\;\;\;\;\;\;\;\; \\
\le r^{\frac{1}{2m}}(\Psi_1\Psi_2^*\Psi_3\Psi_4^*\cdots\Psi_{m-2}\Psi_{m-1}^*\Psi_m\Psi_1^*\Psi_2\Psi_3^*\Psi_4\cdots\Psi_{m-2}^*\Psi_{m-1}\Psi_m^*)\;\;\;\;\;\;\;\;\;\;\;\;\;\;\;\;\;\;\;\;\;
\label{ess_ineq2}
\end{align}
\end{theorem}
\begin{proof}
  If $m$ is even, then we have
$$\left(\left(\Psi_1^{(\frac{1}{m})}\circ\Psi_2^{(\frac{1}{m})}\circ\cdots\circ\Psi_m^{(\frac{1}{m})}\right)^*\left(\Psi_1^{(\frac{1}{m})}\circ\Psi_2^{(\frac{1}{m})}\circ\cdots\circ\Psi_m^{(\frac{1}{m})}\right)\right)^{\frac{m}{2}}=$$
$$\left((\Psi_1^*)^{(\frac{1}{m})}\circ(\Psi_2^*)^{(\frac{1}{m})}\circ\cdots\circ(\Psi_m^*)^{(\frac{1}{m})}\right)\left(\Psi_2^{(\frac{1}{m})}\circ\Psi_3^{(\frac{1}{m})}\circ\cdots\circ\Psi_1^{(\frac{1}{m})}\right)$$
$$\left((\Psi_3^*)^{(\frac{1}{m})}\circ(\Psi_4^*)^{(\frac{1}{m})}\circ\cdots\circ(\Psi_2^*)^{(\frac{1}{m})}\right)\left(\Psi_4^{(\frac{1}{m})}\circ\Psi_5^{(\frac{1}{m})}\circ\cdots\circ\Psi_3^{(\frac{1}{m})}\right)\cdots$$
$$\left((\Psi_{m-1}^*)^{(\frac{1}{m})}\circ(\Psi_m^*)^{(\frac{1}{m})}\circ\cdots\circ(\Psi_{m-2}^*)^{(\frac{1}{m})}\right)\left(\Psi_m^{(\frac{1}{m})}\circ\Psi_1^{(\frac{1}{m})}\circ\cdots\circ\Psi_{m-1}^{(\frac{1}{m})}\right).$$

It follows from (\ref{tool_ess}), (\ref{again}) and \cite[Theorem 3.2(i)]{BP22b} that
%Theorem ...
\begin{align}
\nonumber
\gamma (\Psi_1^{(\frac{1}{m})}\circ\Psi_2^{(\frac{1}{m})}\circ\cdots\circ\Psi_m^{(\frac{1}{m})})^{m}\;\;\;\;\;\;\;\;\;\;\;\;\;\;\;\;\;\;\;\;\;\;\;\;\;\;\;\;\;\;\;\;\;\;\;\;\;\;\;\;\;\;\;\;\;\;\;\;\;\;\;\;\;\;\;\;\;\;\;\;\;\;\;\\
\nonumber
=r\left(\left(\Psi_1^{(\frac{1}{m})}\circ\Psi_2^{(\frac{1}{m})}\circ\cdots\circ\Psi_m^{(\frac{1}{m})}\right)^*\left(\Psi_1^{(\frac{1}{m})}\circ\Psi_2^{(\frac{1}{m})}\circ\cdots\circ\Psi_m^{(\frac{1}{m})}\right)\right)^{\frac{m}{2}}\;\;\;\;\\
\label{vmesna1}
\le r(\Sigma)\le r(\Psi_1^*\Psi_2\Psi_3^*\Psi_4\cdots\Psi_{m-1}^*\Psi_m)^{\frac{1}{m}}r(\Psi_2^*\Psi_3\Psi_4^*\Psi_5\cdots\Psi_m^*\Psi_1)^{\frac{1}{m}}\cdots\\
\nonumber
r(\Psi_{m-1}^*\Psi_m\Psi_1^*\Psi_2\cdots\Psi_{m-3}^*\Psi_{m-2})^{\frac{1}{m}}r(\Psi_m^*\Psi_1\Psi_2^*\Psi_3\cdots\Psi_{m-2}^*\Psi_{m-1})^{\frac{1}{m}}\;\;\;\\
\nonumber
=r^{\frac{1}{2}}(\Psi_1^*\Psi_2\Psi_3^*\Psi_4\cdots\Psi_{m-1}^*\Psi_m)r^{\frac{1}{2}}(\Psi_1\Psi_2^*\Psi_3\Psi_4^*\cdots\Psi_{m-1}\Psi_m^*)\;\;\;\;\;\;\;\;\;\;\;\;\;\;\\
\nonumber
(r(\Psi_1^*\Psi_2\Psi_3^*\Psi_4\cdots\Psi_{m-1}^*\Psi_m)r(\Psi_m\Psi_{m-1}^*\cdots\Psi_4\Psi_3^*\Psi_2\Psi_1^*))^{\frac{1}{2}},\;\;\;\;\;\;\;\;\;\;\;\;\;\;\;\;\;
\end{align}
where 
$$\Sigma:=(\Psi_1^*\Psi_2\Psi_3^*\Psi_4\cdots\Psi_{m-1}^*\Psi_m)^{(\frac{1}{m})}\circ(\Psi_2^*\Psi_3\Psi_4^*\Psi_5\cdots\Psi_m^*\Psi_1)^{(\frac{1}{m})}\circ\cdots\circ$$
$$(\Psi_{m-1}^*\Psi_m\Psi_1^*\Psi_2\cdots\Psi_{m-3}^*\Psi_{m-2})^{(\frac{1}{m})}\circ(\Psi_m^*\Psi_1\Psi_2^*\Psi_3\cdots\Psi_{m-2}^*\Psi_{m-1})^{(\frac{1}{m})}$$
which completes the proof of (\ref{ess_ineq1}).

If $m$ is odd, we have
\begin{align}
\nonumber
\left(\left(\Psi_1^{(\frac{1}{m})}\circ\Psi_2^{(\frac{1}{m})}\circ\cdots\circ\Psi_m^{(\frac{1}{m})}\right)^*\left(\Psi_1^{(\frac{1}{m})}\circ\Psi_2^{(\frac{1}{m})}\circ\cdots\circ\Psi_m^{(\frac{1}{m})}\right)\right)^{m}\;\;\;\;\;\;\;\;\;\;\;\;\;\;\;\;\;\;\;\;\;\;\;\;\;\;\;\\
\nonumber
=\left((\Psi_1^*)^{(\frac{1}{m})}\circ(\Psi_2^*)^{(\frac{1}{m})}\circ\cdots\circ(\Psi_m^*)^{(\frac{1}{m})}\right)\left(\Psi_2^{(\frac{1}{m})}
\circ\cdots\circ\Psi_m^{(\frac{1}{m})}\circ\Psi_1^{(\frac{1}{m})}\right)\;\;\;\;\;\;\;\;\;\;\;\;\;\;\;\;\;\;\\
\nonumber
\left((\Psi_3^*)^{(\frac{1}{m})}\circ(\Psi_4^*)^{(\frac{1}{m})}\circ\cdots\circ(\Psi_2^*)^{(\frac{1}{m})}\right)\left(\Psi_4^{(\frac{1}{m})}\circ\Psi_5^{(\frac{1}{m})}\circ\cdots\circ\Psi_3^{(\frac{1}{m})}\right)\cdots\;\;\;\;\;\;\;\;\;\;\;\;\;\;\;\;\;\;\\
\nonumber
\left((\Psi_{m-2}^*)^{(\frac{1}{m})}\circ(\Psi_{m-1}^*)^{(\frac{1}{m})}\circ\cdots\circ(\Psi_{m-3}^*)^{(\frac{1}{m})}\right)\left(\Psi_{m-1}^{(\frac{1}{m})}\circ\Psi_m^{(\frac{1}{m})}\circ\cdots\circ\Psi_{m-2}^{(\frac{1}{m})}\right)\;\;\;\;\;\;\;\;\;\\
\nonumber
\left((\Psi_m^*)^{(\frac{1}{m})}\circ(\Psi_1^*)^{(\frac{1}{m})}\circ\cdots\circ(\Psi_{m-1}^*)^{(\frac{1}{m})}\right)\left(\Psi_1^{(\frac{1}{m})}\circ\Psi_2^{(\frac{1}{m})}\circ\cdots\circ\Psi_{m-1}^{(\frac{1}{m})}\circ\Psi_m^{(\frac{1}{m})}\right)\;\;\;\;\;\;\\
\nonumber
\left((\Psi_2^*)^{(\frac{1}{m})}\circ(\Psi_3^*)^{(\frac{1}{m})}\circ\cdots\circ(\Psi_1^*)^{(\frac{1}{m})}\right)\left(\Psi_3^{(\frac{1}{m})}\circ\Psi_4^{(\frac{1}{m})}\circ\cdots\circ\Psi_1^{(\frac{1}{m})}\circ\Psi_2^{(\frac{1}{m})}\right)\cdots\;\;\;\;\;\;\;\\
\nonumber
\left((\Psi_{m-1}^*)^{(\frac{1}{m})}\circ(\Psi_{m}^*)^{(\frac{1}{m})}\circ\cdots\circ(\Psi_{m-2}^*)^{(\frac{1}{m})}\right)\left(\Psi_m^{(\frac{1}{m})}\circ\Psi_1^{(\frac{1}{m})}\circ\cdots\circ\Psi_{m-2}^{(\frac{1}{m})}\circ\Psi_{m-1}^{(\frac{1}{m})}\right)\\
%\nonumber
%\Omega:=(\Psi_1^*\Psi_2\Psi_3^*\Psi_4\cdots\Psi_{m-1}\Psi_m^*\Psi_1\Psi_2^*\Psi_3\cdots\Psi_{m-1}^*\Psi_m)^{(\frac{1}{m})}\circ\;\;\;\;\;\;\;\;\;\;\;\;\;\;\;\;\;\;\;\;\;\;\;\;\;\;\;\;\;\;\;\;\;\\
%\nonumber
%(\Psi_2^*\Psi_3\Psi_4^*\Psi_5\cdots\Psi_m\Psi_1^*\Psi_2\Psi_3^*\cdots\Psi_m^*\Psi_1)^{(\frac{1}{m})}\circ\cdots\circ\;\;\;\;\;\;\;\;\;\;\;\;\;\;\;\;\;\;\;\;\;\;\;\;\;\;\;\;\;\;\;\;\;\;\;\;\;\\
%\nonumber
%(\Psi_m^*\Psi_1\Psi_2^*\Psi_3\cdots\Psi_{m-1}^*\Psi_m\Psi_1^*\Psi_2\cdots\Psi_{m-2}^*\Psi_{m-1})^{(\frac{1}{m})}\;\;\;\;\;\;\;\;\;\;\;\;\;%\;\;\;\;\;\;\;\;\;\;\;\;\;\;\;\;\;\;\;\;\;
\end{align}
It follows  from (\ref{tool_ess}), (\ref{again}) and \cite[Theorem 3.2(i)]{BP22b} 
that
\begin{align}
\nonumber
\gamma (\Psi_1^{(\frac{1}{m})}\circ\Psi_2^{(\frac{1}{m})}\circ\cdots\circ\Psi_m^{(\frac{1}{m})})^{2m}=\;\;\;\;\;\;\;\;\;\;\;\;\;\;\;\;\;\;\;\;\;\;\;\;\;\;\;\;\;\;\;\;\;\;\;\;\;\;\;\;\;\;\;\;\;\;\;\;\;\;\;\;\;\;\;\;\;\;\;\;\\
\nonumber
r\left(\left(\Psi_1^{(\frac{1}{m})}\circ\Psi_2^{(\frac{1}{m})}\circ\cdots\circ\Psi_m^{(\frac{1}{m})}\right)^*\left(\Psi_1^{(\frac{1}{m})}\circ\Psi_2^{(\frac{1}{m})}\circ\cdots\circ\Psi_m^{(\frac{1}{m})}\right)\right)^m\le r(\Omega) \le \;\;\\
\nonumber
r(\Psi_1^*\Psi_2\Psi_3^*\Psi_4\cdots\Psi_{m-1}\Psi_m^*\Psi_1\Psi_2^*\Psi_3\Psi_4^*\cdots\Psi_{m-1}^*\Psi_m)=\;\;\;\;\;\;\;\;\;\;\;\;\;\;\;\;\;\;\;\;\;\;\;\;\;\;\;\;\\
\nonumber
r(\Psi_1\Psi_2^*\Psi_3\cdots\Psi_{m-1}^*\Psi_m\Psi_1^*\Psi_2\cdots\Psi_m^*)\;\;\;\;\;\;\;\;\;\;\;\;\;\;\;\;\;\;\;\;\;\;\;\;\;\;\;\;\;\;\;\;\;\;\;\;\;\;\;\;\;\;\;\;\;\;\;\;\;\;\;\;\;
\end{align}
where
$$\Omega:=(\Psi_1^*\Psi_2\Psi_3^*\Psi_4\cdots\Psi_{m-1}\Psi_m^*\Psi_1\Psi_2^*\Psi_3\cdots\Psi_{m-1}^*\Psi_m)^{(\frac{1}{m})}\circ$$
$$(\Psi_2^*\Psi_3\Psi_4^*\Psi_5\cdots\Psi_m\Psi_1^*\Psi_2\Psi_3^*\cdots\Psi_m^*\Psi_1)^{(\frac{1}{m})}\circ\cdots\circ$$
$$(\Psi_m^*\Psi_1\Psi_2^*\Psi_3\cdots\Psi_{m-1}^*\Psi_m\Psi_1^*\Psi_2\cdots\Psi_{m-2}^*\Psi_{m-1})^{(\frac{1}{m})}$$
which proves (\ref{ess_ineq2}).
\end{proof}
The following corollary follows from (\ref{ess_ineq1}) and  (\ref{vmesna1}).
\begin{corollary} Let $\Psi$ and $\Sigma$ be bounded sets of positive kernel operators on $L^2(X, \mu)$  and let  $r \in\{\rho _{ess},\hat{\rho} _{ess}\}$. Then
\begin{eqnarray}
\nonumber
\gamma (\Psi ^{(\frac{1}{2})} \circ  \Sigma^{(\frac{1}{2})} ) &\le& r  \left ( (\Psi^* \Sigma) ^{(\frac{1}{2})}\circ (\Sigma ^* \Psi)^{(\frac{1}{2})}\right)^{\frac{1}{2}}\le r (\Psi^* \Sigma )^{\frac{1}{2}} =
r  \left (\Psi \Sigma ^* \right)^{\frac{1}{2}}  . %\\
%\nonumber
%& \le & \|AB ^* \| ^{\frac{1}{2}} \le  \|A\| ^{\frac{1}{2}} \|B \| ^{\frac{1}{2}},
\end{eqnarray}
\end{corollary}
To our  knowledge even the following singelton set case (which as an essential version of (\ref{Pep19})-\cite[Theorem 4.4, (4.8)]{P17+}) is new. 
\begin{corollary} Let $A$ and $B$ be positive kernel operators on $L^2(X, \mu)$. Then
\begin{eqnarray}
\nonumber
\gamma (A^{(\frac{1}{2})} \circ  B^{(\frac{1}{2})} ) &\le&\rho _{ess}  \left ( (A^* B) ^{(\frac{1}{2})}\circ (B^* A)^{(\frac{1}{2})}\right)^{\frac{1}{2}} \le \rho _{ess}  (A^* B ) ^{\frac{1}{2}}=
\rho _{ess}  \left (A B^* \right)^{\frac{1}{2}} . %\\
%\nonumber
%& \le & \|AB ^* \| ^{\frac{1}{2}} \le  \|A\| ^{\frac{1}{2}} \|B \| ^{\frac{1}{2}},
\end{eqnarray}
\label{ess_star}
\end{corollary}
The following result is an essential version of \cite[Theorem 3.3]{B23+} and is proved in a similar way as Theorem \ref{first_ess}. It follows from  (\ref{tool_ess}), (\ref{again}) and Theorem \ref{finally_ess}. To avoid too much repetition of ideas, the details of the proof are omitted.
\begin{theorem}
\label{thkate_ess}
Let $\Psi_1, \ldots , \Psi_m$ be bounded sets of nonnegative matrices that define operators on $l^2$ and let $\alpha \ge \frac{1}{m}$ and $r \in \{\rho _{ess},\hat{\rho}_{ess}\}$.

If $m$ is even, then
\begin{align}
\nonumber
\gamma (\Psi_1^{(\alpha)}\circ\Psi_2^{(\alpha)}\circ\cdots\circ\Psi_m^{(\alpha)})\le r^{\frac{1}{m}}(\Sigma_{\alpha})\le(r(\Psi_1^*\Psi_2\Psi_3^*\Psi_4\cdots\Psi_{m-1}^*\Psi_m)\\
r(\Psi_m\Psi_{m-1}^*\cdots\Psi_4\Psi_3^*\Psi_2\Psi_1^*))^{\frac{\alpha}{2}},
\label{alpha1_ess}
\end{align}
where
\begin{align}
\nonumber
\Sigma_{\alpha}=(\Psi_1^*\Psi_2\Psi_3^*\Psi_4\cdots\Psi_{m-1}^*\Psi_m)^{(\alpha)}\circ(\Psi_2^*\Psi_3\Psi_4^*\Psi_5\cdots\Psi_m^*\Psi_1)^{(\alpha)}\circ\cdots\circ\\
\nonumber
(\Psi_{m-1}^*\Psi_m\Psi_1^*\Psi_2\cdots\Psi_{m-3}^*\Psi_{m-2})^{(\alpha)}\circ(\Psi_m^*\Psi_1\Psi_2^*\Psi_3\cdots\Psi_{m-2}^*\Psi_{m-1})^{(\alpha)}
\end{align}

If $m$ is odd then
\begin{align}
\nonumber
\gamma (\Psi_1^{(\alpha)}\circ\Psi_2^{(\alpha)}\circ\cdots\circ\Psi_m^{(\alpha)}) \le r^{\frac{1}{2m}}(\Omega_{\alpha})\le\;\;\;\;\;\;\;\;\;\;\;\;\;\;\;\;\;\;\;\\
r^{\frac{\alpha}{2}}(\Psi_1\Psi_2^*\Psi_3\Psi_4^*\cdots\Psi_{m-2}\Psi_{m-1}^*\Psi_m\Psi_1^*\Psi_2\cdots\Psi_{m-1}\Psi_m^*)
\label{alpha2_ess}
\end{align}
where
\begin{align}
\nonumber
\Omega_{\alpha}=(\Psi_1^*\Psi_2\Psi_3^*\Psi_4\cdots\Psi_{m-1}\Psi_m^*\Psi_1\Psi_2^*\Psi_3\cdots\Psi_{m-1}^*\Psi_m)^{(\alpha)}\circ\\
\nonumber
(\Psi_2^*\Psi_3\Psi_4^*\Psi_5\cdots\Psi_{m-1}^*\Psi_m\Psi_1^*\Psi_2\Psi_3^*\cdots\Psi_m^*\Psi_1)^{(\alpha)}\circ\cdots\circ\\
\nonumber
(\Psi_m^*\Psi_1\Psi_2^*\Psi_3\Psi_4^*\cdots\Psi_{m-1}^*\Psi_m\Psi_1^*\Psi_2\cdots\Psi_{m-2}^*\Psi_{m-1})^{(\alpha)}.
\end{align}
\end{theorem}

\medskip

The following corollary follows from (\ref{alpha1_ess}).
\begin{corollary} Let $\Psi$ and $\Sigma$ be bounded sets of nonnegative matrices that defined  operators on $l^2$, let  $r \in\{\rho _{ess},\hat{\rho} _{ess}\}$ and $\alpha \ge \frac{1}{2}$. Then
\begin{eqnarray}
\nonumber
\gamma (\Psi ^{(\alpha)} \circ  \Sigma^{(\alpha)} ) &\le& r  \left ( (\Psi^* \Sigma) ^{(\alpha)}\circ (\Sigma ^* \Psi)^{(\alpha)}\right)^{\frac{1}{2}}\le r (\Psi^* \Sigma )^{\alpha} =
r  \left (\Psi \Sigma ^* \right)^{\alpha}  . %\\
%\nonumber
%& \le & \|AB ^* \| ^{\frac{1}{2}} \le  \|A\| ^{\frac{1}{2}} \|B \| ^{\frac{1}{2}},
\end{eqnarray}
\label{dobra_ess10}
\end{corollary}
The following result is an essential version of \cite[Corollary 3.17]{B23+}. It follows from Corollary \ref{dobra_ess10} and Proposition \ref{dog}.
\begin{corollary}
Let $\alpha\ge\frac{1}{2}$ and let $\Psi$ and $\Sigma$ be bounded sets of nonnegative matrices that define operators on $l^2$. If $r\in\{\rho_{ess}, \hat{\rho}_{ess}\}$, then
$$\gamma(\Psi^{(\alpha)}\circ\Sigma^{(\alpha)})\le r((\Psi^*\Sigma)^{(\alpha)}\circ(\Sigma^*\Psi)^{(\alpha)})^{\frac{1}{2}}$$
\be
\nonumber
\le r((\Psi^*\Sigma)^{(\alpha)}\circ(\Psi^*\Sigma)^{(\alpha)})^{\frac{1}{2}}\le r(\Psi^*\Sigma)^{\alpha}.
\label{matrix}
\ee
\end{corollary}
\medskip

Again, to our  knowledge even the following singelton set case (which as an essential version of \cite[Theorem 4.4, (4.9)]{P17+}) is new. 
\begin{corollary} Let $A$ and $B$ be nonnegative matrices that defined  operators on $l^2$  and $\alpha \ge \frac{1}{2}$. Then
\begin{eqnarray}
\nonumber
\gamma (A^{(\alpha)} \circ  B^{(\alpha)} ) &\le&\rho _{ess}  \left ( (A^* B) ^{(\alpha)}\circ (B^* A)^{(\alpha)}\right)^{\frac{1}{2}} \\
\nonumber
&\le&\rho _{ess}  \left ( (A^* B) ^{(\alpha)}\circ (A^* B)^{(\alpha)}\right)^{\frac{1}{2}}
\le \rho _{ess}  (A^* B ) ^{\alpha}=
\rho _{ess}  \left (A B^* \right)^{\alpha}  .%\\
%\nonumber
%& \le & \|AB ^* \| ^{\frac{1}{2}} \le  \|A\| ^{\frac{1}{2}} \|B \| ^{\frac{1}{2}},
\end{eqnarray}
\end{corollary}
We conclude the article by stating additional results that are essential versions of \cite[Theorems 3.5 and 3.6, Corollary 3.7, Theorems 3.8, 3.11 and 3.13, Corollary 3.15]{B23+}, respectively. The results follow from (\ref{tool_ess}), %(\ref{again})
 \cite[Theorem 3.2(i)]{BP22b} and Theorem \ref{finally_ess} and are proved in a similar way than results in \cite{B23+}. To avoid repetition of ideas we omit the details of the proof.
\begin{theorem}
\label{kety_ess}
Let m be odd and let $\Psi_1, \ldots, \Psi_m$ be bounded sets of positive kernel operators on $L^2(X, \mu)$. For $r\in\{\rho_{ess}, \hat{\rho}_{ess}\}$ we have
\begin{align}
\nonumber
\gamma (\Psi_1^{(\frac{1}{m})}\circ\cdots\circ\Psi_m^{(\frac{1}{m})})\le\;\;\;\;\;\;\;\;\;\;\;\;\;\;\;\;\;\;\;\;\;\;\;\;\;\;\;\;\;\;\;\;\;\;\;\;\;\;\;\;\;\;\;\;\;\;\;\;\;\;\;\;\;\;\;\;\;\;\;\;\;\;\;\;\;\;\;\;\;\;\;\;\;\;\;\;\;\;\;\;\;\\
\nonumber
	r((\Psi_1\Psi_2^*)^{(\frac{1}{m})}\circ\cdots\circ
(\Psi_{m}\Psi_1^*)^{(\frac{1}{m})}\circ(\Psi_2\Psi_3^*)^{(\frac{1}{m})}\circ\cdots\circ(\Psi_{m-1}\Psi_{m}^*)^{(\frac{1}{m})})^{\frac{1}{2}}\le\;\;\;\;\;\;\;\;\;\\
%\nonumber
r(\Omega_1^{(\frac{1}{m})}\circ\cdots\circ\Omega_m^{(\frac{1}{m})})^{\frac{1}{2m}}\le r(\Psi_1\Psi_2^*
\cdots\Psi_{m-2}\Psi_{m-1}^*\Psi_m\Psi_1^*
\cdots\Psi_{m-2}^*\Psi_{m-1}\Psi_m^*)^{\frac{1}{2m}},
\label{laufen_ess}
\end{align}
where 
\begin{align}
\nonumber
\Omega_j=\Psi_{2j-1}\Psi_{2j}^*
\cdots\Psi_{m-2}\Psi_{m-1}^*\Psi_m\Psi_1^*\Psi_2\Psi_3^*\cdots\Psi_{m-1}\Psi_{m}^*\Psi_1\Psi_2^*\cdots\Psi_{2j-3}\Psi_{2j-2}^*
\end{align}
for\; $1\le j\le \frac{m-1}{2}$, and
\begin{align}
\nonumber
\Omega_{\frac{m+1}{2}}=\Psi_m\Psi_1^*\Psi_2\Psi_3^*\cdots\Psi_{m-1}\Psi_{m}^*\Psi_1\Psi_2^*\Psi_3\Psi_4^*\cdots\Psi_{m-2}\Psi_{m-1}^*,\;\;\;\;\;\;\;\;\;\;\;\;\;\;\;\;\;\;\;\;\;\;\;\;
\end{align}
\begin{align}
\nonumber
\Omega_j=\Psi_{2j-m-1}\Psi_{2j-m}^*
\cdots\Psi_{m-1}\Psi_m^*\Psi_1\Psi_2^*\Psi_3\Psi_4^*\cdots\Psi_m\Psi_1^*\cdots\Psi_{2j-m-3}\Psi_{2j-m-2}^*
\end{align}
for\; $\frac{m+3}{2}\le j \le m$.
\end{theorem}

\medskip

\begin{theorem}
\label{kate_ess2}
Let $m\in\NN$ be odd and let $\Psi_1, \ldots , \Psi_m$ be bounded sets of nonnegative matrices that define operators on $l^2.$ If $\alpha\ge\frac{1}{m}$ and if  $\Omega_1, \ldots , \Omega_m$ are sets defined in Theorem \ref{kety_ess}, then for $r\in\{\rho_{ess}, \hat{\rho}_{ess}\}$
\begin{align}
\nonumber
\gamma (\Psi_1^{(\alpha)}\circ\Psi_2^{(\alpha)}
\circ\cdots\circ\Psi_m^{(\alpha)})\le\;\;\;\;\;\;\;\;\;\;\;\;\;\;\;\;\;\;\;\;\;\;\;\;\;\;\;\;\;\;\;\;\;\;\;\;\;\;\;\;\;\;\;\;\;\;\;\;\;\;\;\;\;\;\;\;\;\;\;\;\;\;\;\;\;\;\;\;\;\;\;\;\;\;\;\;\;\;\;\;\;\;\\
\nonumber
r((\Psi_1\Psi_2^*)^{(\alpha)}\circ\cdots\circ(\Psi_{m-2}\Psi_{m-1}^*)^{(\alpha)}\circ(\Psi_m\Psi_1^*)^{(\alpha)}
\circ\cdots\circ(\Psi_{m-1}\Psi_m^*)^{(\alpha)})^{\frac{1}{2}}\le\;\;\;\;\;\;\;\;\;\;\;\;\;\;\\
%\nonumber
r(\Omega_1^{(\alpha)}\circ\cdots\circ\Omega_m^{(\alpha)})^{\frac{1}{2m}}\le r(\Psi_1\Psi_2^*\Psi_3\Psi_4^*\cdots\Psi_{m-2}\Psi_{m-1}^*\Psi_m\Psi_1^*\Psi_2\Psi_3^*\cdots\Psi_{m-1}\Psi_m^*)^{\frac{\alpha}{2}}.\;\;\;\;\;
\label{urlaub_ess}
\end{align}
\end{theorem}

\medskip

\begin{corollary}
\noindent (i) Let $\Psi$ and $\Sigma$ be bounded sets of positive kernel operators on $L^2(X, \mu)$ and $r\in\{\rho_{ess}, \hat{\rho}_{ess}\}$. Then
\begin{align}
\nonumber
\gamma (\Psi^{(\frac{1}{3})}\circ(\Sigma^*)^{(\frac{1}{3})}\circ\Psi^{(\frac{1}{3})})
\le r((\Psi^*\Sigma^*)^{(\frac{1}{3})}\circ(\Psi^*\Psi)^{(\frac{1}{3})}\circ(\Sigma\Psi)^{(\frac{1}{3})})^{\frac{1}{2}}\le\;\;\;\;\;\;\;\;\;\;\;\;\;\;\;\;\;\;\;\;\;\;\;\;\;\;\\
r((\Psi^*\Sigma^*\Psi^*\Psi\Sigma\Psi)^{(\frac{1}{3})}\circ(\Psi^*\Psi\Sigma\Psi\Psi^*\Sigma^*)^{(\frac{1}{3})}\circ(\Sigma\Psi\Psi^*\Sigma^*\Psi^*\Psi)^{(\frac{1}{3})})^{\frac{1}{6}}\le \gamma (\Psi\Sigma\Psi)^{\frac{1}{3}}.\;\;\;\;\;\;\;\;\;
\label{henne_ess}
\end{align}
\noindent (ii) If $\Psi$ and $\Sigma$ are bounded sets of nonnegative matrices that define operators on $l^2(R)$ and if $\alpha\ge\frac{1}{3}$ then
\begin{align}
\nonumber
\gamma(\Psi^{(\alpha)}\circ(\Sigma^*)^{(\alpha)}\circ\Psi^{(\alpha)})\le r((\Psi^*\Sigma^*)^{(\alpha)}\circ(\Psi^*\Psi)^{(\alpha)}\circ(\Sigma\Psi)^{(\alpha)})^{\frac{1}{2}}\le\;\;\;\;\;\;\;\;\;\;\;\;\;\;\;\;\;\;\;\;\;\;\;\;\;\;\;\;\;\;\;\;\;\;\;\;\;\;\;\;\;\;\;\;\;\;\;\;\;\;\;\;\;\;\;\;\;\;\;\\
r((\Psi^*\Sigma^*\Psi^*\Psi\Sigma\Psi)^{(\alpha)}\circ(\Psi^*\Psi\Sigma\Psi\Psi^*\Sigma^*)^{(\alpha)}\circ(\Sigma\Psi\Psi^*\Sigma^*\Psi^*\Psi)^{(\alpha)})^{\frac{1}{6}}\le\gamma (\Psi\Sigma\Psi)^{\alpha}.\;\;\;\;\;\;\;\;\;\;\;\;\;\;\;\;\;\;\;\;\;\;\;\;\;\;\;\;\;\;\;\;\;\;\;\;\;\;\;\;\;\;
\label{huhn_ess}
\end{align}
\end{corollary}

%\medskip

Let $S_m$ denote the group of permutations of the set $\{1, \ldots , m \}$. 

\begin{theorem}
\label{zivalska_ess}
 Let m be even, $\tau, \nu\in S_m$, 
  and let $\Psi_1, \ldots , \Psi_m$ be bounded sets of positive kernel operators on $L^2(X, \mu).$ Denote $\Sigma_j=\Psi_{\tau(2j-1)}^*\Psi_{\tau(2j)}$ and $\Sigma_{\frac{m}{2}+j}=\Psi_{\tau(2j)}^*\Psi_{\tau(2j-1)}=\Sigma_{j}^*$ for $j=1, \ldots , \frac{m}{2}$. Let $\Omega_i=\Sigma_{\nu(i)}\cdots\Sigma_{\nu(m)}\Sigma_{\nu(1)}\cdots\Sigma_{\nu(i-1)}$ for $i=1, \ldots , m$ and $r\in\{\rho_{ess}, \hat{\rho}_{ess}\}$.
  
\noindent  (i) Then
\begin{align}
\nonumber
\gamma( \Psi_1^{(\frac{1}{m})}\circ\cdots\circ\Psi_m^{(\frac{1}{m})}) \le r(\Sigma_1^{(\frac{1}{m})}\circ\cdots\circ\Sigma_m^{(\frac{1}{m})})^{\frac{1}{2}}\;\;\;\;\;\;\;\;\;\;\;\;\;\\
\le r(\Omega_1^{(\frac{1}{m})}\circ\cdots\circ\Omega_m^{(\frac{1}{m})})\le r^{\frac{1}{2m}}(\Sigma_{\nu(1)}\cdots\Sigma_{\nu(m)})^{\frac{1}{2m}}.
\label{ziv1}
\end{align}
\noindent (ii) If $\Psi_1, \ldots , \Psi_m$ are bounded sets of nonnegative matrices that define operators on $l^2(R)$ and if $\alpha\ge\frac{1}{m}$, then
\begin{align}
\nonumber
\gamma (\Psi_1^{(\alpha)}\circ\cdots\circ\Psi_m^{(\alpha)}) \le r(\Sigma_1^{(\alpha)}\circ\cdots\circ\Sigma_m^{(\alpha)})^{\frac{1}{2}}\;\;\;\;\;\;\;\;\;\;\;\;\;\;\\
\le r(\Omega_1^{(\alpha)}\circ\cdots\circ\Omega_m^{(\alpha)})^{\frac{1}{2m}} \le r(\Sigma_{\nu(1)}\cdots\Sigma_{\nu(m)})^{\frac{\alpha}{2}}.
\label{ziv2}
\end{align}
\end{theorem}

\medskip

\begin{theorem}
Let $m\in\NN$ be even, $\alpha\ge\frac{2}{m}$, $\tau\in S_m$ and let $\Psi_1, \ldots , \Psi_m$ be bounded sets of nonnegative matrices that define operators on $l^2$. Let $\Sigma_j$ for $j=1, \ldots , m$ be as in Theorem \ref{zivalska_ess} and denote $\Theta_i=\Sigma_i\cdots\Sigma_{\frac{m}{2}}\Sigma_1\cdots\Sigma_{i-1}$ for $i=1, \ldots , \frac{m}{2}$. If $r\in\{\rho_{ess}, \hat{\rho}_{ess}\}$, then
$$\gamma(\Psi_1^{(\alpha)}\circ\cdots\circ\Psi_m^{(\alpha)}) \le r(\Sigma_1^{(\alpha)}\circ\cdots\circ\Sigma_m^{(\alpha)})^{\frac{1}{2}}\le r(\Sigma_1^{(\alpha)}\circ\cdots\circ\Sigma_{\frac{m}{2}}^{(\alpha)})$$
$$=r((\Psi_{\tau(1)}^*\Psi_{\tau(2)})^{(\alpha)}\circ(\Psi_{\tau(3)}^*\Psi_{\tau(4)})^{(\alpha)}\circ\cdots\circ(\Psi_{\tau(m-1)}^*\Psi_{\tau(m)})^{(\alpha)})$$
\be
\le r(\Theta_1^{(\alpha)}\circ\Theta_2^{(\alpha)}\circ\cdots\circ\Theta_{\frac{m}{2}}^{(\alpha)})^{\frac{2}{m}}\le
r(\Psi_{\tau(1)}^*\Psi_{\tau(2)}\Psi_{\tau(3)}^*\Psi_{\tau(4)}\cdots\Psi_{\tau(m-1)}^*\Psi_{\tau(m)})^{\alpha}
\label{zivalska_ess2}
\ee
\end{theorem}

\medskip

\begin{theorem}
\label{zivalska_ess3}
	Let $\Psi_1, \ldots , \Psi_m$ be bounded sets of positive kernel operators on $L^2(X, \mu)$ and $\tau, \nu\in S_m$. Denote  $\Omega_j=\Psi_{\tau(j)}^*\Psi_{\nu(j)}\cdots\Psi_{\tau(m)}^*\Psi_{\nu(m)}
	\cdots\Psi_{\tau(j-1)}^*\Psi_{\nu(j-1)}$ for $j=1, \ldots , m$. Let $r\in\{\rho_{ess}, \hat{\rho}_{ess}\}$.

\noindent  (i) Then
$$\gamma(\Psi_1^{(\frac{1}{m})}\circ\cdots\circ\Psi_m^{(\frac{1}{m})})\le r((\Psi_{\tau(1)}^*\Psi_{\nu(1)})^{(\frac{1}{m})}\circ\cdots\circ(\Psi_{\tau(m)}^*\Psi_{\nu(m)})^{(\frac{1}{m})})^{\frac{1}{2}}$$
\be
\le r((\Omega_1)^{(\frac{1}{m})}\circ\cdots\circ(\Omega_m)^{(\frac{1}{m})})^{\frac{1}{2m}}\le r(\Psi_{\tau(1)}^*\Psi_{\nu(1)}\cdots\Psi_{\tau(m)}^*\Psi_{\nu(m)})^{\frac{1}{2m}}.
\label{ziv3}
\ee
\noindent (ii) If $\Psi_1, \ldots , \Psi_m$ are bounded sets of nonnegative matrices that define operators on $l^2$ and if $\alpha\ge\frac{1}{m}$, then
$$\gamma(\Psi_1^{(\alpha)}\circ\cdots\circ\Psi_m^{(\alpha)})\le r((\Psi_{\tau(1)}^*\Psi_{\nu(1)})^{(\alpha)}\circ\cdots\circ(\Psi_{\tau(m)}^*\Psi_{\nu(m)})^{(\alpha)})^{\frac{1}{2}}$$
\be
\le r(\Omega_1^{(\alpha)}\circ\cdots\circ\Omega_m^{(\alpha)})^{\frac{1}{2m}}\le r(\Psi_{\tau(1)}^*\Psi_{\nu(1)}\cdots\Psi_{\tau(m)}^*\Psi_{\nu(m)})^{\frac{\alpha}{2}}
\label{ziv4}
\ee 
\end{theorem}

\begin{corollary}
Let $m$ be odd and let $\Psi_1, \ldots , \Psi_m$ be bounded sets of positive kernel operators on $L^2(X, \mu)$. Let $\Omega_j$ for $j=1, \ldots, m$ be as in Theorem \ref{zivalska_ess3} and let $r\in\{\rho_{ess}, \hat{\rho}_{ess}\}$.

\noindent  (i) Then
$$\gamma(\Psi_1^{(\frac{1}{m})}\circ\cdots\circ\Psi_m^{(\frac{1}{m})})$$
$$\le r((\Psi_1^*\Psi_2)^{(\frac{1}{m})}\circ\cdots\circ(\Psi_{m-2}^*\Psi_{m-1})^{(\frac{1}{m})}\circ(\Psi_m^*\Psi_1)^{(\frac{1}{m})}\circ(\Psi_2^*\Psi_3)^{(\frac{1}{m})}\circ\cdots\circ$$
$$(\Psi_{m-1}^*\Psi_{m})^{(\frac{1}{m})})^{\frac{1}{2}}\le r(\Omega_1^{(\frac{1}{m})}\circ\cdots\circ\Omega_m^{(\frac{1}{m})})^{\frac{1}{2m}}$$
$$\le r(\Psi_1^*\Psi_2\cdots\Psi_{m-2}^*\Psi_{m-1}\Psi_m^*\Psi_1\Psi_2^*\Psi_3\cdots\Psi_{m-1}^*\Psi_{m})^{\frac{1}{2m}}$$
\be
\nonumber
=r(\Psi_1\Psi_2^*\Psi_3\cdots\Psi_{m-1}^*\Psi_{m}\Psi_1^*\Psi_2\cdots\Psi_{m-2}^*\Psi_{m-1}\Psi_m^*)^{\frac{1}{2m}}.
\label{sonce}
\ee

\noindent (ii) If $\Psi_1, \ldots , \Psi_m$ are nonnegative matrices that define operators on $l^2(R)$ and if $\alpha\ge\frac{1}{m}$, then
$$\gamma(\Psi_1^{(\alpha)}\circ\cdots\circ\Psi_m^{(\alpha)})$$
$$\le r((\Psi_1^*\Psi_2)^{(\alpha)}\circ\cdots\circ(\Psi_{m-2}^*\Psi_{m-1})^{(\alpha)}\circ(\Psi_m^*\Psi_1)^{(\alpha)}\circ(\Psi_2^*\Psi_3)^{(\alpha)}\circ$$
$$\cdots\circ(\Psi_{m-1}^*\Psi_{m})^{(\alpha)})^{\frac{1}{2}} \le r(\Omega_1^{(\alpha)}\circ\cdots\circ\Omega_m^{(\alpha)})^{\frac{1}{2m}}$$
$$\le r(\Psi_1^*\Psi_2\cdots\Psi_{m-2}^*\Psi_{m-1}\Psi_m^*\Psi_1\Psi_2^*\Psi_3\cdots\Psi_{m-1}^*\Psi_{m})^{\frac{\alpha}{2}}$$
\be
\nonumber
=r(\Psi_1\Psi_2^*\Psi_3\cdots\Psi_{m-1}^*\Psi_{m}\Psi_1^*\Psi_2\cdots\Psi_{m-2}^*\Psi_{m-1}\Psi_m^*)^{\frac{\alpha}{2}}.
\label{veter}
\ee 
\end{corollary}

\bigskip

%\begin{remark}{\rm}
%\end{remark}
\noindent {\bf Acknowledgements.}
The authors acknowledge  mobility support by the Slovenian Research and Innovation Agency (Slovenia-USA bilateral project BI-US/22-24-046).

 The second author acknowledges a partial support of  the Slovenian Research and Innovation Agency (grants P1-0222 %, J1-8133 
 and J2-2512).

\bibliographystyle{amsplain}

\end{document}